\documentclass[reqno,11pt]{amsart}
\usepackage[colorlinks=true, linkcolor=blue, citecolor=blue]{hyperref}

\usepackage{amssymb}
\usepackage{amsmath, graphicx, rotating}
\usepackage{color}     
\usepackage{soul}
\usepackage[dvipsnames]{xcolor}   
\usepackage{tcolorbox}

\usepackage{ifthen}
\usepackage{xkeyval}
\usepackage{todonotes}
\usepackage[T1]{fontenc}
\usepackage{lmodern}
\usepackage[english]{babel}

\usepackage{ upgreek }
\usepackage{stmaryrd}
\SetSymbolFont{stmry}{bold}{U}{stmry}{m}{n}
\usepackage{amsthm}
\usepackage{float}

\usepackage{ bbm }
\usepackage{ stmaryrd }
\usepackage{ mathrsfs }
\usepackage{ frcursive }
\usepackage{ comment }

\usepackage{pgf, tikz}
\usetikzlibrary{shapes}
\usepackage{varioref}
\usepackage{enumitem}
\usepackage{longtable}

\usepackage{mathtools}

\usepackage{dsfont}

\definecolor{rouge}{rgb}{0.7,0.00,0.00}
\definecolor{vert}{rgb}{0.00,0.5,0.00}
\definecolor{bleu}{rgb}{0.00,0.00,0.8}

\usepackage[margin=1.2in]{geometry}

\newtheorem{theorem}{Theorem}[section]
\newtheorem*{theorem*}{Theorem}
\newtheorem{lemma}[theorem]{Lemma}

\newtheorem{corollary}[theorem]{Corollary}

\labelformat{hypothesis}{\textbf{M\kern-0.1mm#1}}

\newtheorem{condition}{Condition}
\newtheorem{conditionA}{A\kern-0.1mm}
\labelformat{conditionA}{\textbf{A\kern-0.1mm#1}}

\renewcommand\dots{\hbox to 1em{.\hss.\hss.}}

\theoremstyle{definition}

\numberwithin{equation}{section}

\def\bb#1{\mathbb{#1}}

\def\bf#1{\mathbf{#1}}
\def\scr#1{\mathscr{#1}}

\def\geq{\geqslant}
\def\leq{\leqslant}

\newcommand\ee{\varepsilon}

\DeclareMathOperator{\supp}{supp}

\DeclareMathOperator{\Leb}{Leb}

\def\geq{\geqslant}
\def\leq{\leqslant}
\def\Rd {\bb R^d}
\def\Rd*{(\bb R^d)^*}
\def\Pd{{\mathbb{P}}^{d-1}}
\def\Pd*{(\mathbb{P}^{d-1})^*}

\def\bb#1{\mathbb{#1}}

\begin{document}
\title[Conditioned local limit theorems]{Conditioned local limit theorems for products of positive random matrices} 

\author{Ion Grama}
\author{Hui Xiao}

\curraddr[Grama, I.]{Univ Bretagne Sud, CNRS UMR 6205, LMBA, Vannes, France.}
\email{ion.grama@univ-ubs.fr}


\curraddr[Xiao, H.]{Academy of Mathematics and Systems Science, Chinese Academy of Sciences, Beijing 100190, China.}
\email{xiaohui@amss.ac.cn}

\date{\today }
\subjclass[2020]{ Primary 60F17, 60J05, 60J10. Secondary 37C30 }
\keywords{Products of positive random matrices, exit time, random walk conditioned to stay positive, conditioned local limit theorem, harmonic function, duality.}

\begin{abstract}
Let $(g_{n})_{n\geq 1}$ be a sequence of independent and identically distributed positive random $d\times d$ matrices, 
where $d\geq 2$ is an integer.
For any starting point $x \in \mathbb{R}_+^d$ with $|x| = 1$ and $y \in \mathbb R$, we define the exit time 
$\tau_{x, y} = \inf \{ k \geq 1: y + \log |g_k \cdots g_1 x| < 0 \}$.
In this paper, we investigate the conditioned local probability 
$\mathbb{P} (y + \log |g_n \cdots g_1 x| \in z + [0, \Delta], \tau_{x, y} > n)$ under various assumptions on $y$, $z$ and $\Delta$.
For the case where $z = O(\sqrt{n})$, we establish 
an exact asymptotic result as $n \to \infty$, uniformly in $y$ and $\Delta$,
which extends the classical Caravenna conditioned local limit theorem to the case of products of positive random matrices. 
Our proof does not rely on the reversibility techniques.
Furthermore, for arbitrary $z \in \bb R_+$, 
we deduce a uniform upper bound with rate $n^{-3/2}$. 
\end{abstract}

\maketitle

\section{Introduction and main results}

\subsection{Background and motivation}  
For any integer $d \geq 2$, equip $\bb R^d$ with the standard inner product $\langle \cdot, \cdot \rangle$ 
and the vector norm $|\cdot |$ as $|x| = \sum_{i = 1}^d |\langle x, e_i \rangle|$ for $x \in \bb R^d$,   
where $(e_i)_{1\leq i\leq d}$ is the standard orthonormal basis of $\bb R^d$. 
Let $\bb S_+^{d-1} = \{x \in \bb R^{d}_+, | x | = 1\}$, where $\bb R^{d}_+$ is the positive quadrant of $\bb R^d$. 
A $d \times d$ matrix $g$ is called allowable,
if all the entries of $g$ are nonnegative, 
and every row and every column of $g$ contains at least one strictly positive entry. 
Denote by $\mathcal M_+$ the multiplicative semigroup of $d\times d$ allowable matrices
and let $\mu$ be a Borel probability measure on $\mathcal M_+$.

Let $(g_{n})_{n\geq 1}$ be a sequence of  independent and identically distributed  random matrices
defined on a probability space $(\Omega,\mathscr{F},\mathbb{P})$ with common law $\mu$. 
For any starting point $x \in \bb S_+^{d-1}$, let $S_0^x = 0$ and 
\begin{align*}
  S_n^x = \log |g_n \cdots g_1 x|, \quad  n \geq 1. 
\end{align*}
For $x \in \bb S_+^{d-1}$ and $y \in \bb R$, 
let $\tau_{x, y}$ be the first time when the Markov walk $(y + S_k^x)_{k \geq 0}$ becomes negative:
\begin{align*}
\tau_{x, y} = \inf \{ k \geq 1: y + S_k^x < 0 \}. 
\end{align*}
We are interested in analyzing the local probability expressed by:
\begin{align}\label{intro-local-probab}
\bb P \left( y + S_n^x \in z+ [0, \Delta ], \tau_{x,y} > n \right),
\end{align}
where $y\in \bb R$, $z\geq 0$ and $\Delta > 0$ may vary depending on $n$.

In the case where $S_n^x$ is a sum of independent and identically distributed random variables,
many authors have investigated the asymptotic behavior of this probability. We refer, for example, to 
L\'{e}vy \cite{Levy37},    Borovkov \cite{Borovk62, Borovkov04a, Borovkov04b},   Feller \cite{Fel64}, Spitzer \cite{Spitzer},
Bolthausen \cite{Bolth}, Iglehart \cite{Igle74}, Eppel \cite{Eppel-1979},  Bertoin and Doney \cite{BertDoney94},  
Caravenna \cite{Carav05},  
Doney \cite{Don12},  Vatutin and Wachtel \cite{VatWacht09}, 
Kersting and Vatutin \cite{KV17}. 
More recent development can be found in Denisov, Tarasov and Wachtel \cite{DTW24, DTW24b}.  
This type of results is quite useful, particularly in the study of branching processes in random environments
 and branching random walks, see for example \cite{Aid13, AS14, KV17, GMX22, Peigne Pham 2024EJP, GMX25}. 
There are two primary difficulties in extending the local limit theorems to dependent cases. The first arises from the inapplicability of Wiener-Hopf factorization in this context. The second difficulty is associated with the reversibility techniques, which rely on the exchangeability of independent and identically distributed random variables. Feller's work \cite[Chapter XII, section 2]{Fel64} provides detailed insights into this. 
The issue with dependent random variables is that, in general, the reversed random walk $S_k^{x, *} = S_k^x - S_n^x$, $0\leq k\leq n$, does not exhibit the same dependence structure as the direct walk $(S_k^x)_{k \geq 0}$. Consequently, it cannot be treated in the same manner. For example, this is often the case with Markov chains, where the reversed walk typically does not retain the Markov dependence.

An alternative to Wiener-Hopf factorization for random walks with independent increments in $\mathbb{R}^d$ was introduced by Denisov and Wachtel \cite{Den Wachtel 2015}. This method has shown particular utility in dealing with dependent random variables, as evidenced by Grama, Lauvergnat and Le Page \cite{GLL20}, who established a conditioned local limit theorem for additive functionals of finite Markov chains.
For exact asymptotics in the setting of hyperbolic dynamical systems, we refer to Grama, Quint and Xiao \cite{GQX21}. In the case of products of invertible random matrices, precise asymptotics 
in the conditioned local limit theorem--including a ballot-type theorem and Caravenna-type asymptotics--were derived 
by Grama, Quint and Xiao \cite{GQX24a, GQX24b}. However, definitive results for positive matrices remain open.
Recently, Peign\'e and Pham \cite{Peigne Pham 2023} obtained a Caravenna-type asymptotic and bounds for a ballot-type theorem. Yet, their results lack uniformity in the starting point $y$ and interval length $\Delta$, a requirement in certain applications, such as extremal position and Seneta-Heyde scaling for branching random walks considered in \cite{GMX22, GMX25}.

This paper aims to bridge this gap by developing methods to ensure uniformity in $y, z$ and $\Delta$
for conditioned local limit theorems involving products of positive random matrices.
Our first objective is to prove an extension of the Caravenna-type theorem. The second objective is to establish a uniform upper bound in the context of the ballot-type theorem.
Our approach relies on heat kernel techniques introduced in the preprint \cite{GX25},
which differ significantly from those used in \cite{Peigne Pham 2023}. 

Building on the methods of \cite{GQX24a, GQX24b}, we expect to extend these techniques in future work to derive exact asymptotics
in the case of positive matrices.

\subsection{Notations and conditions}

For $g \in \mathcal M_+$, let $\|g\| = \sup_{x \in \bb S_+^{d-1}} |gx|$
be the matrix norm. 
For any allowable matrix $g \in \mathcal M_+$ and $x \in \bb S_+^{d-1}$, 
we write $g \cdot x = \frac{gx}{|gx|}$ for the action of $g$ on $x \in \bb S_+^{d-1}$. 
Set $\iota(g) = \inf_{x\in \bb S_+^{d-1} }|gx|$ and $N(g) = \max \{ \|g\|, \iota(g)^{-1} \}.$ 
Denote by $\mathcal M_+^{\circ} \subset \mathcal M_+$ the subsemigroup of $d \times d$ matrices with strictly positive entries. 
The subsemigroup $\mathcal M_+^{\circ}$ is an ideal of $\mathcal M_+$, i.e., 
for any $g \in \mathcal M_+^{\circ}$ and $g' \in \mathcal M_+$, it holds that $g g' \in \mathcal M_+^{\circ}$ and $g' g \in \mathcal M_+^{\circ}$. 

We shall use the following contraction condition 
which says that, with positive probability,
for some $n \geq 1$, the random matrix product $g_n \cdots g_1$ becomes a strictly positive matrix. 

\begin{conditionA}\label{Condi-AP}
$\bb P ( \bigcup_{n \geq 1} (g_n \cdots g_1 \in \mathcal M_+^{\circ}) ) > 0$. 
\end{conditionA}
Note that, by \cite[Lemma 2.1]{HH01}, condition \ref{Condi-AP} is equivalent to 
$\bb P ( \bigcup_{n \geq 1} (g_n \cdots g_1 \in \mathcal M_+^{\circ}) ) = 1$.

We will need the following moment condition of order $2+\delta$.

\begin{conditionA}\label{CondiMoment}
There exists a constant $\delta>0$  such that 
$\int_{\mathcal M_+} (\log N(g) )^{2+\delta} \mu(dg)  < \infty$.
\end{conditionA}

It is well known (cf.\ \cite{Hen97}) that, 
under the first moment condition $\int_{\mathcal M_+} \log N(g) \mu(dg) < \infty$ and condition \ref{Condi-AP}, 
the following strong law of large numbers for $S_n^x$ holds:
for any $x \in \bb S_+^{d-1}$, 
\begin{align*}
\lim_{ n \to \infty } \frac{ S_n^x }{n}  = \lambda,  \quad  \bb P\mbox{-a.s.} 
\end{align*}
where $\lambda \in \bb R$ is a constant (independent of $x$) called the first Lyapunov exponent of $\mu$. 
The central limit theorem is also established in \cite{Hen97}:
under \ref{Condi-AP} and the second moment condition $\int_{\mathcal M_+} (\log N(g) )^{2} \mu(dg) < \infty$,
for any $x \in \bb S_+^{d-1}$, 
the normalised sum $\frac{S_n^x - n \lambda }{ \sqrt{n} }$ converges in distribution to the Gaussian law 
with mean $0$ and variance $\sigma^2$ given by
\begin{align*}
\sigma^2  = \lim_{n\to\infty} \frac{1}{n} \bb E \left[ (S_n^x - n \lambda )^{2}  \right], 
\end{align*}
where $\sigma^2$ does not depend on $x$.

We will assume that the first Lyapunov exponent $\lambda$ equals zero.
\begin{conditionA}\label{Condi-Lya}
$\lambda = 0$.  
\end{conditionA}

For any starting point $x \in \bb S_+^{d-1}$, let $X_0^x = x$ and $X_n^x = (g_n \cdots g_1) \cdot x$ for $n \geq 1$. 
It is known (see e.g. \cite[Theorem 2.1]{HH08}) that, 
under \ref{Condi-AP}, 
the Markov chain $(X_n^x)_{n \geq 0}$ has a unique invariant probability measure $\nu$ on $\bb S_+^{d-1}$.
Denote by $\supp \nu$ the support of the measure $\nu$.

We say that the measure $\mu$ is {\em arithmetic}, if there is $t>0$ together with $\theta \in [0, 2\pi)$ 
and a continuous function 
$\varphi: \bb S_+^{d-1} \to \bb R$ such that for all $n \geq 1$, $g \in \supp \mu^{\otimes n}$ 
and $x \in \supp\nu$, 
\begin{align*}
e^{ i t  \log |gx|} \varphi(g \cdot x) =  e^{i n \theta} \varphi(x).
\end{align*}
In other words, $\log |g x|$ is contained in $\frac{2 \pi}{t} \bb Z$ 
up to a shift that may depend on $g$ and $x$ through the function $\varphi$. 
We need the following non-arithmeticity condition on $\mu$. 

\begin{conditionA}\label{CondiNonarith}
The measure $\mu$ is not arithmetic. 
\end{conditionA}

This condition is used to prove the ordinary local limit theorem 
and we refer to \cite[Chapter X]{HH01} for more details. 
In particular, condition \ref{CondiNonarith} implies that $\sigma>0$.

For a random variable and a Borel set, we write $\bb E (X; B) = \bb E(X \mathds 1_B)$. 
It is shown in \cite{GLP17, Pham18} that the function 
\begin{align*}
V(x, y) : = \lim_{n \to \infty} \bb E  \Big( y + S_n^x;  \tau_{x, y} > n \Big)
\end{align*}
is well defined and non-negative on $\bb S_+^{d-1} \times \bb R$. 
By \cite{GLP17, Pham18}, if there exists a constant $c_0 > 0$ such that 
\begin{align}\label{condition-positivity-V}
\inf_{x \in \bb S_+^{d-1}} \bb P (S_1^x > c_0) >0, 
\end{align}
then the function $V$ is strictly positive on $\bb S_+^{d-1} \times \bb R_+$. 
Moreover, $V$ is harmonic for the Markov walk $(y+S_n^x)_{n\geq 0}$ killed at leaving $\bb R_+$,
see \cite{GLP17, Pham18} for details.

We also need to introduce similar quantities for the dual Markov walk.
For $1 \leq i \leq n$, we denote $h_i = g_{n-i+1}^{ {\rm T} }$, where $g^{ {\rm T} }$ is the transpose of the matrix $g$. 
For $x \in \bb S_+^{d-1}$, let $X_0^{x,*} = x$ and $S_0^{x,*} = 0$, and for $n \geq 1$,  
\begin{align}\label{def-Xn-Sn-dual}
X_n^{x,*} =  (h_n \cdots h_1) \cdot x,  \quad   S_n^{x, *} = - \log |h_n \cdots h_1 x|. 
\end{align}
For any $x \in \bb S_+^{d-1}$ and $y \geq 0$, let 
\begin{align}\label{def-tau-y-star}
\tau_{x, y}^* = \inf \{ k \geq 1: y + S_k^{x, *} < 0 \}. 
\end{align}
The dual harmonic function 
$V^*(x, y) : = \lim_{n \to \infty} \bb E  \big( y + S_n^{*, x};  \tau^*_{x, y} > n \big)$
 is well defined and non-negative on $\bb S_+^{d-1} \times \bb R$.

To establish a ballot-type conditioned local limit theorem with rate $n^{-3/2}$, 
we need the following Furstenberg-Kesten condition from \cite{FK60}. 
For $1 \leq i, j \leq d$, denote by $g^{i,j} = \langle e_i, g e_j \rangle$ the element of the $i$-th row and $j$-th column of the matrix $g$.  

\begin{conditionA}\label{Condi-FK} 
There exists a constant $\varkappa > 1$ such that  for $\mu$-almost every $g \in \mathcal M_+$,
\begin{align}\label{Inequality-FK60}
 \frac{\max_{1\leq i, j\leq d}  g^{i,j} }{ \min_{1\leq i,j\leq d}  g^{i,j} } \leq \varkappa.
\end{align}
\end{conditionA}

Now we introduce a Banach space to be used to state our main results. 
The space $\mathbb{S}^{d-1}_{+}$ is endowed with the Hilbert cross-ratio metric $\mathbf{d}$:  
$\mathbf{d}(x, x') = \frac{ 1- m(x, x')m(x', x) }{ 1 + m(x, x')m(x', x) }$ 
 for $x, x'  \in \bb S_+^{d-1}$, where 
 $m(x, x') = \min \left\{  \frac{ \langle x, e_i \rangle }{ \langle x', e_i \rangle },  i=1,\ldots, d  \right\}.$ 
Denote by $\mathscr{C}(\bb S_+^{d-1})$ the space of real-valued continuous functions on $\bb S_+^{d-1}$, 
equipped with the norm $\|\varphi\|_{\infty}: =  \sup_{x\in \bb S_+^{d-1}} |\varphi(x)|$. 
Set  
\begin{align*}
\|\varphi\|_{ \scr B }: =  \|\varphi\|_{\infty} + 
  \sup_{x, x' \in \bb S_+^{d-1}: x \neq x'} \frac{|\varphi(x) - \varphi(x')|}{ \bf{d}(x, x') }.  
\end{align*}
The Banach space is defined as $\mathscr{B}:= \left\{ \varphi\in \mathscr{C}(\bb S_+^{d-1}): \|\varphi\|_{ \scr B } < \infty \right\}.$

Consider the following function: 
\begin{align} \label{Def-Levydens}
\psi(y,z)=\frac{1}{\sqrt{2\pi }}
   \left( e^{-\frac{(y-z)^2}{2}}- e^{-\frac{(y+z)^2}{2}} \right),  \quad  y, z \in \bb R.   
\end{align} 
Note that $\psi(y,z)>0$ for any $y, z>0$, and that $\psi(y,0)=0$ for any $y\in \bb R$. 
Let $\Phi$ denote the standard normal distribution function on $\bb R$: 
$\Phi(y) = \frac{1}{\sqrt{2 \pi}} \int_{-\infty}^y e^{- \frac{t^2}{2}} dt$, $y \in \bb R$. 
Consider the normalization factor $H(y): = \int_{\bb R_+} \psi(y, z) dz=2 \Phi( y ) -1$, $y \in \bb R$, 
and set, for any $y,z\in \bb R$, 
\begin{align} \label{def func L-001}
\ell(y,z) =\frac{\psi(y,z)}{H(y)}. 
\end{align}
Then, for any $y\in \bb R$, the function $z\mapsto\ell(y,z)$ is a probability density on $\bb R_+$.
It is shown in \cite{GX21, GX25} that uniformly in $z$ 
over compact sets of $\bb R_+$, 
\begin{align} \label{limit of func ell-001}
\lim_{y\to 0} \ell(y,z) = \lim_{y\to 0} \frac{\psi(y,z)}{H(y)} = \phi^+(z),
\end{align}
where $\phi^+(z) = z e^{-z^{2}/2} \mathds 1_{\{ z \geq 0 \} }$, $z \in \bb R$, is the Rayleigh density function.

For $y\in \bb R$, denote
\begin{align} \label{def-Lx-001}
L(y) = \frac{H(y)}{y},
\end{align}
and  we define $L$ at $x=0$ by continuity: 
$L(0) := \lim_{y\to 0} L(y) = \frac{2}{\sqrt{2\pi}}.$
Note that, $L$ is positive and even on $\bb R$, i.e.\ $L(y)>0$ and $L(y)=L(-y)$, for any $y\in \bb R$. 
For $n \geq 1$, $x\in \bb S_+^{d-1}$ and $y \in \bb R$, set 
\begin{align*}
V_n(x, y) = V(x,y) L\left(\frac{y}{\sigma\sqrt{n}}\right), \quad V^*_n(x, y) = V^*(x,y) L\left(\frac{y}{\sigma\sqrt{n}}\right). 
\end{align*}
It is easy to show that, for any $n \geq 1$, $x \in \bb S_+^{d-1}$ and $y \in \bb R$, 
\begin{align}\label{inequ-Vn-001}
V_n(x, y) \leq c \sqrt{n}, 
\quad 
V^*_n(x, y)  \leq c \sqrt{n}. 
\end{align}

In the sequel $c$ denotes a positive constant and $c_{\alpha}, c_{\alpha, \beta}$ denote positive constants 
depending only on their indices.  
All these constants are subject to change their values every occurrence.

\subsection{Main results}\label{sec-CLLT}

Our first result is concerned with a conditioned local limit theorem of Gnedenko type for products of positive random matrices. 
It is deduced from upper and lower bounds established in Theorem \ref{t-B 002}, 
see Section \ref{sec-proof Th CondLLT Version1}.

\begin{theorem} \label{Thm-nodrift-ysmall}
Assume \ref{Condi-AP}, \ref{Condi-Lya}, \ref{CondiNonarith}, and that \ref{CondiMoment} holds for $\delta\geq 1$.  
Let $\Delta_0 >0$ be a fixed real number. 
Then, there exists a constant $\beta>0$ such that, as $n \to \infty$, uniformly in $x \in \bb S_+^{d-1}$, $y \in \bb R$, $z \in \bb R_+$ and $\Delta \geq \Delta_0$, 
\begin{align}\label{CLLT-thm1-interval}
\bb P \Big(   &  y + S_n^x \in [z, z+\Delta],\, \tau_{x, y} > n   \Big) \notag \\
&=  (1 + o(1)) \frac{ V_n(x, y)  }{ \sigma^2 n}   
 \int_{z}^{z+\Delta}\ell \left( \frac{y}{\sigma\sqrt{n}},  \frac{z'}{\sigma\sqrt{n}} \right) dz'   +  \Delta \frac{1+ V_n(x, y) }{n} O(n^{-\beta}).  
\end{align}
Moreover, under the Furstenberg-Kesten condition \ref{Condi-FK} (in place of \ref{Condi-AP}), 
the requirement $\delta \geq 1$ can be weakened to $\delta > 0$. 
\end{theorem}

As a first application, we derive the following extension of the Caravenna-type result:  as $n \to \infty$, 
uniformly for $x \in \bb S_+^{d-1}$, $y, z, \Delta \in \mathbb{R}$ satisfying $\frac{y}{\sqrt{n}} \to 0$, 
$z \in \bb R_+$, and $\Delta \geq \Delta_0$,  we have:
\begin{align} \label{non-lattice Caravenna-result-with-rate-007}
\bb P \Big(  y + S_n^x \in [z, z+\Delta],\, \tau_{x, y} > n   \Big)
& = (1 + o(1))  \Delta \frac{2V(x, y)}{\sqrt{2\pi}\sigma^2 n}  \int_{z}^{z+\Delta} \phi^+\left( \frac{z'}{\sigma \sqrt{n}} \right) dz'  \notag\\
& \quad + \Delta \frac{1 + V(x, y)}{n} \, O \left(n^{-\beta} + \frac{|y|}{\sqrt{n}} \right).
\end{align}
To derive \eqref{non-lattice Caravenna-result-with-rate-007}, we rely on the asymptotic relation
\eqref{limit of func ell-001} and $\lim_{y\to 0} L(y) = \frac{2}{\sqrt{2\pi}}$ 
(more precisely we need to use \cite[Lemma 2.3]{GX-2024-CCLT} and \cite[Lemma 5.4]{GX25}).

As a second application, we derive the following result:  as $n \to \infty$,   
uniformly for $x \in \bb S_+^{d-1}$, $y, z, \Delta \in \mathbb{R}$ satisfying 
$y \to \infty$, $z \in \bb R_+$ and $\Delta \geq \Delta_0$, we have:
\begin{align} \label{non-lattice Caravenna-result-with-rate-007-bb}
\bb P \Big(  y + S_n^x \in [z, z+\Delta],\, \tau_{x, y} > n   \Big)
& =  \frac{ 1 + o(1)  }{ \sigma \sqrt{n}}   
 \int_{z}^{z+\Delta} \psi \left( \frac{y}{\sigma\sqrt{n}},  \frac{z'}{\sigma\sqrt{n}} \right) dz' \notag\\
& \quad  +  \Delta \frac{1+ V_n(x, y) }{n} O(n^{-\beta}).
\end{align}
The proof of \eqref{non-lattice Caravenna-result-with-rate-007-bb} relies on \eqref{def func L-001}, \eqref{inequ-Vn-001}
and the fact that $V(x, y)/y \to 1$ as $y \to \infty$.

Very recently, under the Furtenberg-Kesten condition \ref{Condi-FK}, 
a result similar to \eqref{non-lattice Caravenna-result-with-rate-007}
 was obtained by Peign\'e and Pham \cite{Peigne Pham 2023}. 
 Our result improves over this results in several aspects. 
 Firstly, 
 we do not require the Furtenberg-Kesten condition; 
 instead, we only need a much weaker positivity condition \ref{Condi-AP}. 
 Secondly,  we relax the exponential moment condition used in \cite{Peigne Pham 2023} 
 to the polynomial moments. 
Thirdly, the results in \cite{Peigne Pham 2023} hold only for fixed $y$ and $\Delta$, 
while Theorem \ref{Thm-nodrift-ysmall} allows for $y$, $z$ and $\Delta$ depending on $n$. 
In particular, when $\frac{y}{\sqrt{n}}\to 0$ as $n\to\infty$, our result extends that in \cite{Peigne Pham 2023}.
In the case where the starting point $y$ is proportional to $\sqrt{n}$, 
another asymptotic holds true, as one can see from \eqref{non-lattice Caravenna-result-with-rate-007-bb}. 
It is also important to note that, in contrast to \cite{Peigne Pham 2023}, 
 our method of the proof do not need to employ reversibility, thus allowing for weaker conditions. 

Moreover, our local limit theorem applies to the pair $(X_n^x, S_n^x)$ with general target functions on both components, 
thereby enlarging the domain of applications. 
 Let $\mathcal F$ be the set of real-valued functions $F$ on $\bb S_+^{d-1} \times \bb R$ satisfying
the following two conditions:
for any $y \in \bb R$, the function $x \mapsto F(x, y)$ is Lipschitz continuous on $\bb S_+^{d-1}$ with respect to the distance $\bf d$;   
for any $x \in \bb S_+^{d-1}$, the function $y \mapsto F(x, y)$ is continuous compactly supported on $\bb R$.
 
 \begin{theorem} \label{Thm-nodrift-ysmall-target}
Assume \ref{Condi-AP}, \ref{Condi-Lya}, \ref{CondiNonarith}, and that \ref{CondiMoment} holds for $\delta\geq 1$.  
Let $\Delta_0 >0$ be a fixed real number. 
Then, there exists a constant $\beta>0$ such that, for any $F \in \mathcal F$, 
as $n \to \infty$, 
uniformly in $x \in \bb S_+^{d-1}$, $y \in \bb R$ and $z \in \bb R_+$, 
\begin{align*} 
& \bb E \Big[ F \left(X_n^x, y + S_{n}^x - z \right);  \tau_{x, y} > n \Big]  \notag\\
&  = (1 + o(1)) \frac{ V_n(x, y) }{\sigma^2 n}
 \int_{\bb S_+^{d-1} \times \bb R_+}    F \left(x, z' - z \right)  
   \ell \left(\frac{y}{\sigma\sqrt{n}} ,\frac{z'}{\sigma\sqrt{n}}\right) \nu(dx)  dz' 
  +  \frac{1+ V_n(x, y) }{n} O(n^{-\beta}). 
\end{align*}
Moreover, under the Furstenberg-Kesten condition \ref{Condi-FK} (in place of \ref{Condi-AP}), 
the requirement $\delta \geq 1$ can be weakened to $\delta > 0$. 
\end{theorem}

Our second result gives an uniform upper bound of order $n^{-3/2}$.

\begin{theorem} \label{Thm-CLLT-small-x}
Assume \ref{CondiMoment}, \ref{Condi-Lya}, \ref{CondiNonarith} and \ref{Condi-FK}. 
Let $\Delta_0 >0$ be a fixed real number. 
Then there exists a constant $c>0$ such that for any $n \geq 1$, $x, x' \in \bb S_+^{d-1}$, $y \in \bb R$, 
$z \in \bb R_+$ and $\Delta \geq \Delta_0$, 
\begin{align}\label{uniform-bound-n32}
\bb P \left( y + S_n^x \in  [z, z + \Delta],   \tau_{x, y} > n   \right)
\leq c \frac{ \Delta}{n^{3/2}} \left( 1 + V_n(x, y) \right) \left( 1  + V^*_n(x', z + \Delta) \right). 
\end{align} 
\end{theorem}

If in addition we assume \eqref{condition-positivity-V},
the bound \eqref{uniform-bound-n32} becomes: for any $n \geq 1$, $x, x' \in \bb S_+^{d-1}$, 
$y, z \in \bb R_+$ and $\Delta \geq \Delta_0$, 
\begin{align}\label{uniform-bound-n32-002}
\bb P \left( y + S_n^x \in  [z, z + \Delta],   \tau_{x, y} > n   \right)
\leq c \frac{ \Delta}{n^{3/2}} V_n(x, y)  V^*_n(x', z + \Delta). 
\end{align}

For fixed $y$ and $z$, a  bound similar to \eqref{uniform-bound-n32-002} was established in \cite{Peigne Pham 2023}.
The uniform version plays a key role in several applications, particularly in the study of the extremal position and asymptotic behavior of the derivative martingale for branching random walks on the semigroup of positive matrices--as demonstrated in \cite{GMX22, GMX25}, where the bound \eqref{uniform-bound-n32} is crucially applied.
These applications are among the primary motivations for our work.
Moreover, for such applications, it is essential to derive these results under a change of measure.
The corresponding extensions can be obtained straightforwardly, so we omit the details here.

\subsection{Proof method}
 Our method for establishing Theorem \ref{Thm-nodrift-ysmall} relies on the convolution technique, avoiding the need for the Wiener-Hopf factorization which is typically employed in the case of independent random variables. 
We follow the methodology developed in \cite{GX21, GX25}, 
although the situation differ from the latter because we must consider dependence.
We adapt techniques from \cite{GQX21}, which were originally designed for hyperbolic dynamical systems, but
as such, our approach may be of interest for other models. 

The key idea is to leverage the Markov property to express the local probability \eqref{intro-local-probab} as a limit convolution of two  components. One of these components corresponds to a standard local limit theorem, while the other relates to a conditioned integral limit theorem. It turns out that this convolution yields the main term in \eqref{CLLT-thm1-interval}, 
which stands as one of the principal advancement in our paper. 
This approach could be considered a refinement of methods found in Caravenna \cite{Carav05}, Doney \cite{Don12}, and Denisov and Wachtel \cite{Den Wachtel 2015}.

Another subtle aspect of the proof is the utilization of the reversibility technique. For sums of independent and identically distributed random variables, reversibility plays a significant role in proving Gnedenko-type conditioned local limit theorems (cf. \cite{Carav05, Don12, Den Wachtel 2015}). Typically, it is employed to bound the remainder terms that arise in the convolution-based approach described earlier.
In our paper we demonstrate that a Gnedenko-type conditioned limit theorem \eqref{CLLT-thm1-interval} can be established without relying on the reversibility technique, which constitutes our second advancement. This approach can also be applied to establish a conditioned local limit theorem of Gnedenko type for a more general class of Markov chains.

We will comment now on the obtention of Theorem \ref{Thm-CLLT-small-x}.
One of the key aspects is the essential use of the reversibility technique. 
The problem is that, in the context of products of positive random matrices, the formal reversed walk $(- S_n^x + S_k^x)_{1 \leq k \leq n}$ is no longer a product of random matrices. Consequently, the conditioned limit theory developed earlier does not apply to this situation.
To address this challenge, we replace the formal reversed walk with an approximate reversed walk constructed using the reversed product of random matrices $g_1^{*} \cdots g_n^{*}$, where $g^{*}$ represents the adjoint of the matrix $g$. This approximation becomes feasible due to the Furstenberg-Kesten condition \ref{Condi-FK}, which enables us to establish the following approximate duality relation between the Markov walks $S_n^x = \log |g_n \cdots g_1 x|$ and $S_n^{x, *} = - \log |g_1^{*} \cdots g_n^{*} x|:$
there exists a constant $c \in (0, \infty)$ such that, for any $0 \leq k \leq n$, 
\begin{align}\label{intro-duality}
S_{n-k}^{x, *} - c \leq  - S_n^x + S_k^x  \leq   S_{n-k}^{x, *} + c.  
\end{align}
Theorem \ref{Thm-CLLT-small-x} is obtained by using a succession of results: 
the approximate duality relation \eqref{intro-duality},
 the more general version of the Gnedenko-type conditioned local limit theorem 
given by Theorem \ref{t-B 002}, 
and the conditioned central limit theorem (cf. Theorem \ref{Thm-Lim-tau}) for the reversed Markov walk $(S_n^{x, *})_{n \geq 0}$.

\section{Preliminary results}

\subsection{Martingale approximation and inequalities}

We follow the strategy developed in \cite{HH08, BQ16a} 
to give a martingale approximation for $S_n^x = \log |g_n \cdots g_1 x |$,
under the second moment assumption. 
Let $\mathscr{F}_0$ be the trivial $\sigma$-algebra 
and $\mathscr{F}_n = \sigma \{ g_k: 1 \leq k \leq n \}$.

\begin{lemma}\label{Prop-MartApp}
Assume \ref{Condi-AP}, \ref{Condi-Lya} and that $\int_{\mathcal M_+} (\log N(g) )^{2} \mu(dg)  < \infty$. 
Then, there exist a martingale $(M_n^x, \mathscr{F}_n)_{n \geq 0}$ and a constant $a>0$ such that
\begin{align*}
\sup_{n\geq 0} \sup_{x \in \bb S_+^{d-1} } | S_n^x - M_n^x| \leq  a. 
\end{align*}
\end{lemma}

\begin{proof}
Define the transfer operator $P$ as follows: for any measurable function $\varphi$ on $\bb S_+^{d-1}$,
\begin{align}\label{def-operator-P-01}
P \varphi(x) = \int_{\mathcal M_+}  \varphi(g \cdot x) \mu(dg),  \quad  x \in \bb S_+^{d-1}. 
\end{align}
By \cite{HH08}, under \ref{Condi-AP}, there exist constants $c, c' >0$ such that for any $\varphi \in \scr B$ and $n \geq 1$, 
\begin{align}\label{ergodic-operator-P}
\| P^n \varphi - \nu(\varphi) \|_{ \scr B } \leq  c' e^{-c n} \| \varphi \|_{\scr B},
\end{align}
where $\nu$ satisfies $\mu * \nu = \nu$. 
For $g \in \mathcal M_+$ and $x \in \bb S_+^{d-1}$, let  
\begin{align*}
\rho(g, x) = \log |gx|, 
\end{align*}
which satisfies the cocycle property, i.e. $\rho(g_2 g_1, x) = \rho(g_2, g_1 \cdot x) + \rho(g_1, x)$
for any $g_1, g_2 \in \mathcal M_+$ and $x \in \bb S_+^{d-1}$. 
Consider the expected increase of the cocycle $\rho$: 
\begin{align}\label{Def-Aver-Cocy}
\theta(x) =  \int_{\mathcal M_+} \rho(g, x) \mu(dg),  \quad  x \in \bb S_+^{d-1}.   
\end{align}
The function $\theta$ belongs to $\mathscr{B}$.
Indeed, 
$\| \theta \|_{\infty} \leq \int_{\mathcal M_+} \log N(g) \mu(dg) < \infty$
and 
\begin{align*}
\frac{|\theta(x) - \theta(x')|}{ \bf{d}(x, x') }
& \leq   \frac{\int_{\mathcal M_+}|\rho(g,x) - \rho(g,x')|\mu(dg)}{ \mathbf{d}(x,x')} \mathds{1}_{\{ \mathbf{d}(x,x') > \frac{1}{2}\}} \nonumber\\
& \quad   +  
\frac{\int_{\mathcal M_+}|\rho(g,x)-\rho(g,x')|\mu(dg)}{ \mathbf{d}(x,x')} \mathds{1}_{\{ \mathbf{d}(x,x') \leq \frac{1}{2}\}}.
\end{align*}
The first term is bounded by $4 \int_{\mathcal M_+}\log N(g)\mu(dg)$. 
By \cite[Lemma 3.1]{HH08}, there exists $c>0$ such that $| \rho(g,x) - \rho(g,x') | \leq c \mathbf{d} (x,x')$
for any $g \in \mathcal M_+$ and $x, x' \in \mathbb{S}_{+}^{d-1}$ satisfying $\bf{d} (x,x') \leq 1/2$. 
This shows that the second term is also bounded and hence $\theta \in \mathscr{B}$.

We next prove that there exists a continuous function $\psi$ on $\mathbb{S}_+^{d-1}$ satisfying 
\begin{align}\label{cohomological equ1}
\theta = \psi - P \psi. 
\end{align}
By \cite{HH08}, the Lyapunov exponent $\lambda$ can be rewritten as 
$\lambda = \int_{ \bb S_+^{d-1} } \int_{\mathcal M_+} \log |gx| \mu(dg) \nu(dx).$ 
By condition \ref{Condi-Lya}, we have  
$\nu(\theta) = \lambda = 0$, so that $\theta \in H := {\rm ker}(\nu)\cap \mathscr{B}$.
By \cite{HH08}, $H$ is a closed $P$-invariant subspace of $\mathscr{B}$
and the spectral radius of the restricted operator $P\vert_{H}$ is strictly smaller than $1$.
Therefore, the equation \eqref{cohomological equ1} has a unique solution $\psi =\sum_{n=0}^{\infty} P^n \theta$. 
By $\mu * \nu = \nu$, we have $\nu(P^n)=\nu$ for any $n\geq 0$.
Hence $P^n \theta \in H$ for any $n\geq 0$, so that $\psi \in H$.

Set $\rho_0(g,x) = \rho(g,x)- \psi(x) + \psi(g \cdot x)$ for $g \in \mathcal M_+$ and $x\in \mathbb{S}_{+}^{d-1}$, 
where $\psi$ solves \eqref{cohomological equ1}. 
Since $\rho$ is a cocycle, so is $\rho_0$. 
By \eqref{Def-Aver-Cocy} and \eqref{cohomological equ1}, 
we have $\int_{\mathcal M_+} \rho_0(g, x) \mu(dg) = 0$ for any $x \in \mathbb{S}_{+}^{d-1}$. 
Define $M_0^x = 0$ and $M_n^x = \rho_0(g_n \cdots g_1, x)$ for $n \geq 1$,  
then $(M_n^x, \mathscr{F}_n)_{n \geq 0}$ is a martingale and $S_n^x = M_n^x + \psi(x) - \psi(X_n^x)$, 
which concludes the proof of the lemma. 
\end{proof}

The following inequality is a generalisation of Fuk's inequality for martingales; see Haeusler \cite[Lemma 1]{Hae84} for the proof.

\begin{lemma}[\cite{Hae84}]  \label{Lem_FukNagaev}
Let $\xi_1, \ldots, \xi_n$ be a martingale difference sequence with respect to the non-decreasing 
$\sigma$-fields $\mathscr F_0, \mathscr F_1, \ldots, \mathscr F_n$.
Then, for all $u, v, w > 0$,
\begin{align*}
\bb P \left(  \max_{1 \leq k \leq n} \left| \sum_{i=1}^k \xi_i \right| \geq u \right)
& \leq  \sum_{i=1}^n \bb P \left( |\xi_i|  > v \right)
  + 2 \bb P \left(  \sum_{i=1}^n \bb E \left( \xi_i^2 |  \mathscr F_{i-1} \right) > w \right)   \notag\\
& \quad  + 2 \exp \left\{ \frac{u}{v} \left( 1 - \log \frac{uv}{w} \right) \right\}. 
\end{align*}
\end{lemma}

Using this lemma and the spectral gap property \eqref{ergodic-operator-P}, we establish the following Fuk-type inequality for a target function $\varphi$ on the Markov chain $(X_n^x)_{n \geq 0}$. This inequality will play a key role in proving the lower bound \eqref{eqt-A 002}.
Here, $\gamma$ is taken to be sufficiently large, and it is crucial to carefully track the dependence on the target function $\varphi$.

\begin{lemma} \label{FK-joint-inequality}
Assume \ref{Condi-AP}, \ref{CondiMoment} and \ref{Condi-Lya}.  
Then there exist constants $c, c', c_0>0$ such that for any  $\gamma > c_0$, $n\geq 1$, $x \in \bb S_+^{d-1}$ 
and any nonnegative function $\varphi \in \scr B$,
\begin{align*} 
I :  =    
 \bb E  \bigg[   \varphi \left( X_n^x \right)    
  \mathds 1_{ \left\{ \max_{1 \leq j \leq n } | S_{j}^x |  \geq  \gamma \sqrt{n}   \right\}  }  \bigg]     
    \leq   c' \left( \frac{1}{ \gamma } + \frac{1}{ n^{\delta/2} }  \right)  \nu(\varphi)   
      +  \frac{c'}{ \sqrt{n} } \|\varphi\|_{\infty} + c' e^{-cn^{1/12}} \| \varphi \|_{\mathscr B}. 
\end{align*}
\end{lemma}

\begin{proof}
Set $p = n - [n^{1/12}]$.
We write 
\begin{align*}
I & =    
 \bb E  \bigg[   \varphi \left( X_n^x \right)    
  \mathds 1_{ \left\{ \max_{1 \leq j \leq n } | S_{j}^x |  \geq \gamma \sqrt{n}   \right\}  }   
    \mathds 1_{ \left\{  \max_{ p < j \leq n } | S_{j}^x - S_p^x |  \leq  n^{1/3}  \right\} }  \bigg]     \notag\\ 
  & \quad  +   
 \bb E  \bigg[   \varphi \left( X_n^x \right)    
  \mathds 1_{ \left\{ \max_{1 \leq j \leq n } | S_{j}^x |  \geq \gamma \sqrt{n}   \right\}  }   
    \mathds 1_{ \left\{  \max_{ p < j \leq n } | S_{j}^x - S_p^x |  >  n^{1/3}  \right\} }  \bigg]    \notag\\
  & =: I_1 + I_2. 
\end{align*}
For the first term $I_1$, since  $\max_{1 \leq j \leq n } | S_{j}^x  |  
\leq \max_{1 \leq j \leq p } | S_{j}^x | +  \max_{p < j \leq n } | S_{j}^x - S_p^x |$
and $\gamma \sqrt{n} - n^{1/3} \geq \frac{2 \gamma}{3}  \sqrt{n}$ for any $\gamma \geq 3$ and $n \geq 1$, 
we get that for $\gamma \geq 3$ and $n \geq 1$, 
\begin{align}\label{Lower_F_ee_kkk-002}
I_1 \leq   
 \bb E  \bigg[   \varphi \left( X_n^x  \right)    
  \mathds 1_{ \left\{ \max_{1 \leq j \leq p } | S_{j}^x |  \geq  \frac{2 \gamma}{3}  \sqrt{n}   \right\}  } 
   \bigg].
\end{align}
In view of \eqref{def-operator-P-01}, it holds that $P^k \varphi (x) =  \bb E [ \varphi (X_k^x) ]$ 
for any $k\geq 1$ and $x \in \bb S_+^{d-1}$. 
By the Markov property, from \eqref{Lower_F_ee_kkk-002} and \eqref{ergodic-operator-P}
 it follows that there exist constants $c, c' >0$ such that for any $x \in \bb S_+^{d-1}$ and $n \geq 1$, 
\begin{align}\label{Bound-I1-0a}
I_1  
& \leq  \bb E \Big[  P^{[n^{1/12}]} \varphi  \left( X_p^x \right) 
 \mathds 1_{ \left\{ \max_{1 \leq j \leq p } |  S_{j}^x |  \geq  \frac{2 \gamma}{3}  \sqrt{n}   \right\}  }  \Big]  \notag\\
& \leq  
\nu(\varphi)   \bb P \left( \max_{1 \leq j \leq p } | S_{j}^x  |  \geq  \frac{2 \gamma}{3}  \sqrt{n}   \right)   
+ c' e^{-c n^{1/12}} \| \varphi \|_{\mathscr B}  \notag\\
& \leq  \nu(\varphi)   \bb P \left( \max_{1 \leq j \leq n } | M_{j}^x  |  \geq  \frac{\gamma}{2}  \sqrt{n}   \right)   
+ c' e^{-c n^{1/12}} \| \varphi \|_{\mathscr B}, 
\end{align}
where $(M_j^x, \scr F_j)_{j \geq 1}$ is a martingale from Lemma \ref{Prop-MartApp}. 
By utilizing Lemma \ref{Lem_FukNagaev} with the values $u = \frac{\gamma}{2}  \sqrt{ n}$, 
$v = c_0^2 \sqrt{n}$ and $w = \frac{c_0^2}{8} \gamma n$, 
we get that the the first term is bounded by $c' n^{- \delta/2}$ with $\delta >0$ given in condition \ref{CondiMoment}, the second term is bounded by $\frac{c'}{\gamma}$,
and the third term equals $2 \exp ( - \frac{\gamma}{2c_0^2} \log \frac{4}{e} )$, so that 
\begin{align}\label{appli-Fuk-Nagaev-001}
  \bb P \left(  \max_{1 \leq j \leq n } | M_{j}^x |  \geq  \frac{\gamma}{2}  \sqrt{n}   \right) 
\leq  \frac{c'}{ n^{\delta/2} } + \frac{c'}{ \gamma }  + 2 e^{- c \gamma } \leq c'' \left(  \frac{1}{ \gamma } + \frac{1}{ n^{\delta/2} } \right). 
 \end{align} 
Therefore,  
\begin{align}\label{bound-I1-ww}
I_1 \leq  c'' \left( \frac{1}{ \gamma } + \frac{1}{ n^{\delta/2} }  \right) \nu(\varphi)  + c'' e^{-cn^{1/12}} \| \varphi \|_{\mathscr B}. 
\end{align}
For the second term $I_2$, using Markov's inequality and condition \ref{CondiMoment}, we get 
\begin{align*}
I_2  & \leq  \|\varphi\|_{\infty}  \bb P \left( \max_{ p < j \leq n } | S_{j}^x - S_p^x |  >  n^{1/3} \right)
\leq  \|\varphi\|_{\infty}  \sum_{k=1}^{n-p}  \bb P \left( \log N(G_k)  >  n^{1/3} \right)   \notag\\
& \leq  \|\varphi\|_{\infty}  \sum_{k=1}^{n-p}  \frac{ \bb E [ (\log N(G_k) )^2 ] }{ n^{2/3} }
\leq  \|\varphi\|_{\infty}  \sum_{k=1}^{n-p}  k^2  \frac{ \bb E [ (\log N(g) )^2 ] }{ n^{2/3} }
\leq  \frac{c}{ \sqrt{n} } \|\varphi\|_{\infty}. 
\end{align*}
 This, together with \eqref{bound-I1-ww}, concludes the proof of the lemma.
\end{proof}


\subsection{A conditioned central limit theorem}

Denote by $\Phi^+(t)= (1 - e^{-\frac{t^2}{2}}) \mathds 1_{\{ t \geq 0 \}}$, $t \in \bb R$, the Rayleigh distribution function.
Now we state the following conditioned central limit theorem. 

\begin{theorem}\label{Thm-Lim-tau}
Assume \ref{Condi-AP}, \ref{CondiMoment}  and \ref{Condi-Lya}.    
Then there exists a constant $\ee_0 >0$ such that the following assertions hold. 
\begin{enumerate}[label=\arabic*., leftmargin=*]
\item
For any $\ee \in (0,\ee_0)$, there exists a constant $c_{\ee} >0$ such that 
for any $n \geq 1$, $x \in \bb S_+^{d-1}$, $y \leq n^{1/2-\ee}$ and $t \in \bb R$, 
\begin{align}\label{Condi_Ray02}
\left| \bb{P} \left(\frac{y + S_n^x}{\sigma\sqrt{n}} \leq t,  \,  \tau_{x, y} > n \right)
 - \Phi^{+}(t) \frac{2V(x,y)}{ \sigma \sqrt{2\pi n} }  \right| 
   \leq  c_{\ee}  \frac{  1 + \max\{y, 0\} }{n^{1/2+\ee}}.  
\end{align} 
\item
For any $\ee \in (0,\ee_0)$, there exists a constant $c_{\ee} >0$ such that  for any 
$n\geq 1$, $x \in \bb S_+^{d-1}$, $y \geq n^{\frac{1}{2} -  \ee}$ and $t \geq 0$,    
\begin{align} \label{CorCCLT02}
\left| \bb P \left(  \frac{ y + S_n^x }{ \sigma \sqrt{n}} \leq  t, \tau_{x, y} >n\right) 
     -   \int_{0}^{t}  \psi \left( \frac{y}{\sigma \sqrt{n}},  z  \right)  dz    \right|
   \leq   c_{\ee}  n^{- \ee }.  
\end{align}
\end{enumerate}
\end{theorem}

The proof of Theorem \ref{Thm-Lim-tau}  can be done in a similar way as those in \cite{GX21}
and we omit the details. 

By combining  the asymptotics \eqref{Condi_Ray02} and \eqref{CorCCLT02}, 
one can deduce the following unified version of the conditioned central limit theorem.

\begin{theorem}\label{Theor-probIntUN-002} 
Assume \ref{Condi-AP}, \ref{CondiMoment} and \ref{Condi-Lya}. Then, there exists a constant $\eta_0 >0$ such that for any $\eta \in(0, \eta_0)$, $n \geq 1$, $t\in \bb R_+$, $x\in \bb S_+^{d-1}$ and $y \in \bb R$,
\begin{align*} 
\left| \mathbb{P} \left(  \frac{y+S^x_n }{ \sigma \sqrt{n}} \leq  t, \tau_{x,y} >n\right) 
     -\frac{ V_n(x, y) }{\sigma\sqrt{n}}   \int_{0}^{t}  \ell \left(\frac{y}{\sigma \sqrt{n}}, u   \right)  du    \right|
\leq c_{\eta} \frac{1 + V_n(x, y) }{n^{1/2+\eta}}.
\end{align*}
\end{theorem}

In paricular, taking $t = \infty$, we get that there exists a constant $c>0$ 
such that for any $n \geq 1$, $x \in \bb S_+^{d-1}$ and $y \in \bb R$, 
\begin{align}\label{bound-tau-y-bis}
\bb P (\tau_{x, y} > n) \leq c \frac{ 1 + V_n(x, y) }{ \sqrt{n} }. 
\end{align}

\subsection{A local limit theorem}

Let $\mathscr H$ be the set of real-valued functions $F$ on $\bb S_+^{d-1} \times \bb R$ 
such that the integral $\int_{\bb R} \| F (\cdot, y) \|_{\mathscr B} dy$ is finite; 
for any $y \in \bb R$, the function $x \mapsto F(x, y)$ is Lipschitz continuous on $\bb S_+^{d-1}$ with respect to the distance $\bf d$;
and for any $x \in \bb S_+^{d-1}$, the function $y \mapsto F(x, y)$ is measurable on $\bb R$. 
 By \cite[Lemma 5.3]{GQX21}, the function $F$ is measurable on $\bb S_+^{d-1} \times \bb R$ and 
the function $y \mapsto \| F (\cdot, y) \|_{ \mathscr B }$ is measurable on $\bb R$. 
For any $F \in \mathscr H$,  we use the notation 
\begin{align} \label{def-norms-F}
\| F \|_{\mathscr H} =  \int_{\bb R} \| F(\cdot, y) \|_{\mathscr B}   dy,
\qquad
\| F \|_{\nu \otimes \Leb} =  \int_{ \bb S_+^{d-1} \times \bb R} | F(x, y) |  \nu(dx) dy. 
\end{align}
Clearly, it holds that $\| F \|_{\nu \otimes \Leb} \leq \| F \|_{\mathscr H}$ for any $F \in \mathscr H$. 

For any nonnegative Borel measurable functions $F, H: \bb S_+^{d-1} \times \mathbb R \mapsto \mathbb R_+$ and $\ee>0$, 
we say that $F$ is $\ee$-dominated by $H$ (say $F \leq_{\ee} H$),  if
\begin{align} \label{def_upper_envelope_001}
 F(x, y)  \leq  H(x, y+v), \  \forall x \in \bb S_+^{d-1},  \  \forall y \in \bb R,  \ \forall |v| \leq \ee.   
\end{align}
Let $\phi(y) = \frac{1}{\sqrt{2 \pi}} e^{- \frac{y^2}{2}}$, $y \in \bb R$, be the standard normal density function. 
Now we state the effective version of the local limit theorem for products of positive random matrices.

 \begin{theorem} \label{LLT-general}
Assume \ref{Condi-AP}, \ref{CondiMoment} and \ref{Condi-Lya}.  
Let $\delta \in (0, 1]$ be from \ref{CondiMoment}. 
Then there exists a constant $c>0$ with the following property: 
for any $\ee \in (0,\frac{1}{8})$, there exists a constant $c_{\ee} >0$ such that, 
for any  $x \in \bb S_+^{d-1}$, $y \in \bb R$, $n\geq 1$,  any nonnegative function $F$ and any function $H \in \mathscr H$
satisfying $F\leq_{\ee} H$, 
\begin{align}
  \bb E F \left( X_n^x, y + S_n^x \right)
&  \leq   \frac{1 + c\ee}{\sigma \sqrt{n}}  \int_{\bb S_+^{d-1} \times \bb R } 
  H \left( x', y' \right)  \phi \left(\frac{y' - y}{ \sigma \sqrt{n} } \right)  \nu(dx')  dy'    \notag\\
& \quad +  \frac{c \ee}{n} \| H \|_{\nu \otimes \Leb}  +  \frac{c_{\ee}}{ n^{(1 + \delta)/2} }  \| H \|_{\mathscr H},     \label{LLT-general001}
\end{align}
and for any  $x \in \bb S_+^{d-1}$, $y \in \bb R$, $n\geq 1$, 
any nonnegative  function $F$ and nonnegative  functions $H, M \in \mathscr H$
satisfying $M \leq_{\ee} F \leq_{\ee} H$, 
\begin{align}
  \bb E F \left( X_n^x, y + S_n^x \right)  
 & \geq  \frac{1}{ \sigma \sqrt{n}} \int_{\bb S_+^{d-1} \times \bb R } 
  (M \left( x', y' \right) - c \ee H(x', y') )   \phi \left(\frac{y' - y}{\sigma \sqrt{n}} \right)  \nu(dx')  dy'     \notag\\
& \quad   -  \frac{c \ee}{n} \| H \|_{\nu \otimes \Leb} 
    -  \frac{c_{\ee}}{ n^{(1 + \delta)/2} }
      \left(  \| H \|_{\mathscr H}  +  \| M \|_{\mathscr H} \right).   
       \label{LLT-general002}
\end{align}
\end{theorem}

By using the spectral gap theory (cf.\ \cite{HH01, HH08}) for products of positive random matrices, 
the proof of Theorem \ref{LLT-general} can be carried out in an analogous way as those in \cite{GX21, GQX21}
and hence the details are omitted.


\subsection{Contraction properties}

By \cite{Hen97}, there exists $c(g) \in (0, 1]$ such that for any $x, x' \in \bb S_+^{d-1}$,  
\begin{align}\label{inequality-contraction}
\bf d (g \cdot x, g \cdot x' ) \leq c(g) \mathbf{d}(x, x'), 
\end{align}
and $c(g) < 1$ if and only if $g$ is a strictly positive matrix. 
By \cite{HH01, HH08}, under conditions \ref{Condi-AP} and \ref{CondiMoment}, 
there exist constants $r \in (0, 1)$ and $c >0$  such that for any $x, x' \in \bb S_+^{d-1}$ and $n \geq 1$, 
\begin{align}\label{Contractovity-Markov-chain}
\bb E  \left[ \bf d( X_n^x, X_n^{x'}) \right]
\leq  c  \,  r^n \bf d(x, x'). 
\end{align}
Using \eqref{Contractovity-Markov-chain} we prove the following result.

\begin{lemma}\label{Lem-conti-square}
Assume \ref{Condi-AP} and \ref{CondiMoment}.   
Then there exists a constant $c >0$  such that for any $x, x' \in \bb S_+^{d-1}$ and $n \geq 1$, 
\begin{align}\label{Compare-starting-x01}
\bb E  \Big| S_n^x -  S_n^{x'}  \Big| 
\leq 
\bb E^{1/2}  \Big[ \Big| S_n^x -  S_n^{x'}  \Big|^2  \Big]
\leq  c \,  \bf d(x, x')
\end{align}
and 
\begin{align}\label{Compare-starting-x01-bb}
\bb E  \left[ \left| \min_{1 \leq j \leq n}  S_j^x  -  \min_{1 \leq j \leq n}  S_j^{x'} \right| \right]
\leq 
\bb E^{1/2}  \Big[ \Big| \min_{1 \leq j \leq n}  S_j^x  -  \min_{1 \leq j \leq n}  S_j^{x'}  \Big|^2  \Big]
\leq  c \,  \bf d(x, x'). 
\end{align}
\end{lemma}

\begin{proof}
The first inequalities in \eqref{Compare-starting-x01} and \eqref{Compare-starting-x01-bb} follow directly by Cauchy-Schwarz's inequality. 
We now prove the second one in \eqref{Compare-starting-x01}. 
Since $\log |g_2 g_1 x| = \log |g_2 (g_1 \cdot x)| + \log |g_1 x|$ for any $g_1, g_2 \in \mathcal M_+$
and $x \in \bb S_+^{d-1}$, we get that for any $1 \leq j \leq n$ and $x, x' \in \bb S_+^{d-1}$, 
\begin{align}\label{inequa-Snx1-Snx2}
S_j^x  \leq  S_j^{x'} + \sum_{k=1}^j  \sigma_{\textup{Lip}}(g_k) \bf d \left( X_{k-1}^x, X_{k-1}^{x'} \right), 
\end{align}
where we use the convention $g_0 x = x$ for $x \in \bb S_+^{d-1}$, and denote
\begin{align*}
\sigma_{\textup{Lip}}(g) = \sup_{x, x' \in \bb S_+^{d-1}: x \neq x'} \frac{ |\log |gx| - \log |gx'| | }{ \bf d(x, x') }. 
\end{align*}
By \eqref{inequa-Snx1-Snx2}, we have 
\begin{align*}
\Big| S_n^x -  S_n^{x'}  \Big|
 \leq  \sum_{k=1}^n  \sigma_{\textup{Lip}}(g_k) \bf d \left( X_{k-1}^x, X_{k-1}^{x'} \right). 
\end{align*}
Using the inequality $\bb E^{1/2} [(X+Y)^2] \leq \bb E^{1/2} (X^2) + \bb E^{1/2} (Y^2)$ 
and the fact that $g_k$ is independent of $g_{k-1} \cdots g_1$, we get
\begin{align}\label{inequality-Sn-x-x}
\bb E^{1/2}  \Big[ \Big| S_n^x -  S_n^{x'}  \Big|^2  \Big]
& \leq \bb E^{1/2}  \Big[ \Big| \sum_{k=1}^n  
   \sigma_{\textup{Lip}}(g_k) \bf d \left( X_{k-1}^x, X_{k-1}^{x'} \right) \Big|^2  \Big] \notag\\
& \leq  \sum_{k=1}^n  \bb E^{1/2}  \Big[ \Big| \sigma_{\textup{Lip}}(g_k) \bf d \left( X_{k-1}^x, X_{k-1}^{x'} \right) \Big|^2  \Big]  \notag\\
& =  \sum_{k=1}^n  \bb E^{1/2}  \Big[ \Big| \sigma_{\textup{Lip}}(g_k) \Big|^2  \Big]
\bb E^{1/2}  \Big[ \Big|  \bf d \left( X_{k-1}^x, X_{k-1}^{x'} \right) \Big|^2  \Big]  \notag\\
& \leq  c \sum_{k=1}^n   \bb E^{1/2}  \Big[ \Big|  \bf d \left( X_{k-1}^x, X_{k-1}^{x'} \right) \Big|^2  \Big], 
\end{align}
where in the last inequality we used the fact that $\bb E  \Big[ \Big| \sigma_{\textup{Lip}}(g_k) \Big|^2  \Big] \leq c$ for some constant $c$, 
under condition \ref{CondiMoment}. 
Using  \eqref{inequality-contraction}, we have
\begin{align*}
 \bb E^{1/2}  \Big[ \Big|  \bf d \left( X_{k-1}^x, X_{k-1}^{x'} \right) \Big|^2  \Big]  
& =  \bb E^{1/2}  \Big[ \Big|  \bf d \left( X_{k-1}^x, X_{k-1}^{x'}  \right) \Big|^2  
      \mathds 1_{ \{ \bf d \left( X_{k-1}^x, X_{k-1}^{x'} \right) \leq  \frac{1}{k^2}  \bf d (x, x') \} }   \Big]  \notag\\
& \quad  + \bb E^{1/2}  \Big[ \Big|  \bf d \left( X_{k-1}^x, X_{k-1}^{x'} \right) \Big|^2  
  \mathds 1_{ \{ \bf d \left( X_{k-1}^x, X_{k-1}^{x'}  \right) >  \frac{1}{k^2}  \bf d (x, x') \} }   \Big]  \notag\\
& \leq  \frac{1}{k^2}  \bf d (x, x') 
+  \bf d (x, x')  \bb P^{1/2} \left( \bf d \left( X_{k-1}^x, X_{k-1}^{x'} \right) >  \frac{1}{k^2}  \bf d (x, x')  \right). 
\end{align*}
By Markov's inequality and \eqref{Contractovity-Markov-chain}, 
there exist constants $r \in (0, 1)$ and $c >0$  such that for any $x, x' \in \bb S_+^{d-1}$ and $k \geq 1$, 
\begin{align*}
\bb P \left( \bf d \left( X_{k-1}^x, X_{k-1}^{x'} \right) >  \frac{1}{k^2}  \bf d (x, x')  \right)
 \leq  \frac{k^2}{ \bf d (x, x') }  \bb E  \left[ \bf d \left( X_{k-1}^x, X_{k-1}^{x'} \right) \right] 
 \leq c  k^2  r^{k-1}. 
\end{align*}
Therefore, 
\begin{align*}
\bb E^{1/2}  \Big[ \Big|  \bf d \left( X_{k-1}^x, X_{k-1}^{x'} \right) \Big|^2  \Big]
\leq  \bf d (x, x')  \left( \frac{1}{k^2}  +  c  k  r^{ (k-1)/2 }  \right), 
\end{align*}
so that 
\begin{align}\label{inequa-distance-x-x}
\sum_{k=1}^n   \bb E^{1/2}  \Big[ \Big|  \bf d \left( X_{k-1}^x, X_{k-1}^{x'} \right) \Big|^2  \Big]
\leq  c  \bf d (x, x')  \sum_{k=1}^n  \left( \frac{1}{k^2}  +  c  k  r^{ (k-1)/2 }  \right) 
\leq  c'  \bf d (x, x'), 
\end{align}
which, combined with \eqref{inequality-Sn-x-x}, ends the proof of the second one in \eqref{Compare-starting-x01}. 

We next show the second inequality in \eqref{Compare-starting-x01-bb}. 
By \eqref{inequa-Snx1-Snx2}, for any $1 \leq j \leq n$ and $x, x' \in \bb S_+^{d-1}$, 
\begin{align*}
\min_{1 \leq j \leq n}  S_j^x 
\leq S_j^x 
 \leq  S_j^{x'} + \sum_{k=1}^n  \sigma_{\textup{Lip}}(g_k) \bf d \left( X_{k-1}^x, X_{k-1}^{x'} \right).  
\end{align*}
Taking the minimum over $1 \leq j \leq n$ on the right-hand side, we get for any $x, x' \in \bb S_+^{d-1}$, 
\begin{align}\label{Inequality-minimum-01}
\min_{1 \leq j \leq n} S_j^x 
 \leq  \min_{1 \leq j \leq n}  S_j^{x'}
 + \sum_{k=1}^n  \sigma_{\textup{Lip}}(g_k) \bf d \left( X_{k-1}^x, X_{k-1}^{x'} \right).  
\end{align}
For the lower bound, using the inequality
\begin{align*}
S_j^x 
 \geq  S_j^{x'} - \sum_{k=1}^j  \sigma_{\textup{Lip}}(g_k) \bf d \left( X_{k-1}^x, X_{k-1}^{x'} \right)  
 \geq  S_j^{x'} - \sum_{k=1}^n  \sigma_{\textup{Lip}}(g_k) \bf d \left( X_{k-1}^x, X_{k-1}^{x'} \right), 
\end{align*}
one can check that for any $x, x' \in \bb S_+^{d-1}$ and $n \geq 1$, 
\begin{align}\label{Inequality-minimum-02}
\min_{1 \leq j \leq n}  S_j^x 
 \geq  \min_{1 \leq j \leq n}  S_j^{x'}
  - \sum_{k=1}^n  \sigma_{\textup{Lip}}(g_k) \bf d \left( X_{k-1}^x, X_{k-1}^{x'} \right).  
\end{align}
Combining \eqref{Inequality-minimum-01} and \eqref{Inequality-minimum-02},
we obtain 
\begin{align*}
\bb E^{1/2} \left[ \left| \min_{1 \leq j \leq n}  S_j^x  - \min_{1 \leq j \leq n}  S_j^{x'}  \right|^2 \right] 
 \leq  \bb E^{1/2}  \Big[ \Big| \sum_{k=1}^n  
   \sigma_{\textup{Lip}}(g_k) \bf d \left( X_{k-1}^x, X_{k-1}^{x'} \right) \Big|^2  \Big]  
    \leq  c \,  \bf d(x, x'),
\end{align*}
where in the last inequality we used \eqref{inequality-Sn-x-x} and \eqref{inequa-distance-x-x}. 
This ends the proof of \eqref{Compare-starting-x01-bb}. 
\end{proof}

\begin{lemma}\label{Lem_Integral_LLT}
Assume \ref{Condi-AP} and \ref{CondiMoment}.  
Let $\rho$ be a measurable and compactly supported function on $\bb R$. 
Then, there exists a constant $c = c(\rho) >0$ such that for  
any integrable function $H$ on $\bb R$ and any $n \geq 1$, 
\begin{align}\label{inequa-H-inte-0a}
\int_{\bb R}  \sup_{x \in \bb S_+^{d-1}} \bb E  \Big| H *  \rho \left( y + S_n^x \right) \Big|   dy
\leq   c n^{1/4}   \int_{\bb R} |H(u)| du 
\end{align}
and
\begin{align}\label{inequa-H-inte-0b}
\int_{\bb R}  \sup_{x \in \bb S_+^{d-1}} \bb E^{1/2} \Big[   \Big| H *  \rho \left( y + S_n^x \right) \Big|^2  \Big]   dy
\leq   c n^{ \frac{4 + \delta}{ 8 + 4 \delta } }   \int_{\bb R} |H(u)| du  
\end{align} 
where $\delta >0$ is from the moment condition \ref{CondiMoment}. 
\end{lemma}

\begin{proof}
We first prove \eqref{inequa-H-inte-0a}. 
Let $[-c_0, c_0]$ be the support of the function $\rho$, for some constant $c_0>0$. 
Since 
\begin{align}\label{inequa-convolution-nn}
H *  \rho \left( y + S_n^x \right)
= \int_{\bb R}  \rho \left( y + S_n^x - z' \right) H(z') dz'
=  \int_{\bb R}  \rho \left( z + S_n^x \right) H(y-z) dz, 
\end{align}
by Fubini's theorem, it holds that 
\begin{align}\label{decomp-Exp-H-rho}
\int_{\bb R}  \sup_{x \in \bb S_+^{d-1}} \bb E  \Big| H *  \rho \left( y + S_n^x \right) \Big|   dy
& \leq  \int_{\bb R} \left( \int_{\bb R}  |H(y-z)|  
  \sup_{x \in \bb S_+^{d-1}}  \bb E |\rho \left( z + S_n^x \right)|  dz    \right)  dy  \notag\\
&  =  \int_{\bb R} |H(u)| du 
 \int_{\bb R}   \sup_{x \in \bb S_+^{d-1}}  \bb E |\rho \left( z + S_n^x \right)|  dz  \notag\\
& \leq  \sup_{t \in [-c_0, c_0]}  \rho(t)  \int_{\bb R} |H(u)| du 
\int_{\bb R} \sup_{x\in \bb S_+^{d-1}}  \bb{P}  \left(  z + S_n^x  \in [-c_0, c_0] \right) dz. 
\end{align}
Now we decompose the integral into two parts: 
$\int_{\bb R} \sup_{x\in \bb S_+^{d-1}}  \bb{P}  \left(  z + S_n^x  \in [-c_0, c_0] \right) dz = J_1(n) + J_2(n)$, 
where 
\begin{align*}
& J_1(n) = \int_{|z| \leq 2 c_0 n^{3/4}} 
     \sup_{x\in \bb S_+^{d-1}}  \bb{P}  \left(  z + S_n^x  \in [-c_0, c_0] \right) dz,  \\
& J_2(n) = \int_{|z| > 2 c_0 n^{3/4} } 
     \sup_{x\in \bb S_+^{d-1}}  \bb{P}  \left(  z + S_n^x  \in [-c_0, c_0] \right)  dz. 
\end{align*}

For $J_1(n)$, 
since $\rho$ is supported on $[-c_0, c_0]$, by the local limit theorem (\eqref{LLT-general001} of Theorem \ref{LLT-general}),
there exists a constant $c >0$ such that for any $n \geq 1$, 
\begin{align}\label{LLT-integral}
 \sup_{x\in \bb S_+^{d-1}} \sup_{z\in \mathbb{R}}  \bb{P}  \left(  z + S_n^x  \in [-c_0, c_0] \right)
 \leq  \frac{c}{\sqrt{n}} (c_0 + 1). 
\end{align}
Hence there exists a constant $c = c(\rho) >0$ such that for any $n \geq 1$, 
\begin{align}
J_1(n)  \leq   c n^{1/4}.    \label{Pf_dual_Bound_I1}
\end{align}

For $J_2(n)$, 
we first deal with the case when $z > 2c_0 n^{3/4}$, 
so that it holds that $c_0 - z \leq -\frac{z}{2}$ for any $n \geq 1$.   
Under condition \ref{CondiMoment}, 
by Lemmas 4.3 and 4.4 of \cite{GLP17}, there exists a constant $c>0$ such that for any $n \geq 1$, it holds
$\sup_{x \in \bb S_+^{d-1}} \bb E ( |S_n^x|^{2 + \delta} ) \leq c n^{1+ \frac{\delta}{2}}$,
where $\delta > 0$ is from condition \ref{CondiMoment}. 
Therefore, by Markov's inequality,  
there exists a constant $c>0$ such that for any $x \in \bb S_+^{d-1}$,  $n \geq 1$ and $z > 2c_0 n^{3/4}$, 
\begin{align}\label{second-mom-inequa}
\bb{P}  \left(  z + S_n^x  \in [-c_0, c_0] \right)
  \leq  \bb{P}  \left(  S_n^x \leq  -\frac{z}{2} \right) 
 \leq  \frac{c}{z^{2 + \delta}}  \bb E  \left( |S_n^x|^{2 + \delta} \right)
 \leq  \frac{c}{z^{2 + \delta}} n^{1+ \frac{\delta}{2}}. 
\end{align}
so that 
\begin{align*}
\int_{n^{3/4}}^{\infty}  \sup_{x\in \bb S_+^{d-1}}  \bb{P}  \left(  z + S_n^x  \in [-c_0, c_0] \right)  dz
\leq  \frac{c}{ n^{ \frac{3}{4} (1 + \delta) } }  n^{1+ \frac{\delta}{2}} 
\leq  c n^{1/4}.  
\end{align*}
In the same way, we can show that 
\begin{align*}
\int_{-\infty}^{- n^{3/4}}  \sup_{x\in \bb S_+^{d-1}}  \bb{P}  \left(  z + S_n^x  \in [-c_0, c_0] \right)  dz
\leq  c n^{1/4}. 
\end{align*}
Therefore, there exists a constant $c = c(\rho) > 0$ such that 
\begin{align}
J_2(n) \leq  c n^{1/4}. \label{Pf_dual_Bound_I3}
\end{align}
Substituting \eqref{Pf_dual_Bound_I1} and \eqref{Pf_dual_Bound_I3} into \eqref{decomp-Exp-H-rho} 
concludes the proof of \eqref{inequa-H-inte-0a}. 

We next prove \eqref{inequa-H-inte-0b}. 
By \eqref{inequa-convolution-nn} and Minkowski's inequality for integrals, 
we get
\begin{align*}
\bb E^{1/2} \Big[   \Big| H *  \rho \left( y + S_n^x \right) \Big|^2  \Big]
& =  \bb E^{1/2} \left[  \left( \int_{\bb R}  \rho \left( z + S_n^x \right) H(y-z) dz  \right)^2  \right] \notag\\
& \leq  \int_{\bb R}  \bb E^{1/2} \left( |\rho \left( z + S_n^x \right)|^2  |H(y-z)|^2 \right)  dz  \notag\\
& =  \int_{\bb R}  |H(y-z)|  \bb E^{1/2} \left( |\rho \left( z + S_n^x \right)|^2  \right)  dz.
\end{align*}
Therefore, by Fubini's theorem, we obtain 
\begin{align*}
\int_{\bb R}  \sup_{x \in \bb S_+^{d-1}} \bb E^{1/2} \Big[   \Big| H *  \rho \left( y + S_n^x \right) \Big|^2  \Big]   dy
& \leq  \int_{\bb R}  \left( \int_{\bb R} |H(y-z)| dy \right)  \sup_{x \in \bb S_+^{d-1}}  \bb E^{1/2} \left( |\rho \left( z + S_n^x \right)|^2  \right) dz  \notag\\
& =  \int_{\bb R} |H(u)| du   \int_{\bb R}  \sup_{x \in \bb S_+^{d-1}}  \bb E^{1/2} \left( |\rho \left( z + S_n^x \right)|^2  \right) dz. 
\end{align*}
Since $\rho$ is supported on $[-c_0, c_0]$, we get 
\begin{align*}
 \int_{\bb R}  \sup_{x \in \bb S_+^{d-1}}  \bb E^{1/2} \left( |\rho \left( z + S_n^x \right)|^2  \right) dz
 \leq  c  \int_{\bb R}  \sup_{x \in \bb S_+^{d-1}} \bb P^{1/2} \left( z + S_n^x \in [-c_0, c_0] \right) dz = c (I_1(n) + I_2(n)),
\end{align*}
where, for $a = \frac{3 + \delta}{ 4 + 2 \delta }$ with $\delta > 0$ from condition \ref{CondiMoment}, 
\begin{align*}
& I_1(n) =  \int_{|z| \leq 2 c_0 n^a}  \sup_{x \in \bb S_+^{d-1}} \bb P^{1/2} \left( z + S_n^x \in [-c_0, c_0] \right) dz,  \notag\\
& I_2(n) =  \int_{|z| > 2 c_0 n^a}  \sup_{x \in \bb S_+^{d-1}} \bb P^{1/2} \left( z + S_n^x \in [-c_0, c_0] \right) dz. 
\end{align*}
For $I_1(n)$, by \eqref{LLT-integral}, there exists a constant $c = c(\rho) >0$ such that for any $n \geq 1$,
\begin{align*}
I_1(n) \leq c n^{a-\frac{1}{4}} = c n^{ \frac{4 + \delta}{ 8 + 4 \delta } }. 
\end{align*}
For $I_2(n)$, by \eqref{second-mom-inequa}, 
there exists $c>0$ such that for any $x \in \bb S_+^{d-1}$,  $n \geq 1$ and $z > 2c_0 n^{a}$, 
\begin{align*}
\int_{n^{a}}^{\infty}  \sup_{x\in \bb S_+^{d-1}}  \bb{P}^{1/2}  \left(  z + S_n^x  \in [-c_0, c_0] \right)  dz
\leq  c \int_{n^{a}}^{\infty}  \left(  \frac{1}{z^{2 + \delta}} n^{1+ \frac{\delta}{2}}  \right)^{1/2}  dz
=   c n^{ \frac{1}{2} + \frac{\delta}{4} - \frac{a \delta}{2} } 
= c n^{ \frac{4 + \delta}{ 8 + 4 \delta } }. 
\end{align*}
Similarly, we can show that 
\begin{align*}
\int_{-\infty}^{- n^{a}}  \sup_{x\in \bb S_+^{d-1}}  \bb{P}^{1/2}  \left(  z + S_n^x  \in [-c_0, c_0] \right)  dz
\leq  c n^{ \frac{4 + \delta}{ 8 + 4 \delta } }. 
\end{align*}
Hence $I_1(n) + I_2(n) \leq c n^{ \frac{4 + \delta}{ 8 + 4 \delta } }$, completing the proof of \eqref{inequa-H-inte-0b}. 
\end{proof}

As a consequence of Lemmas \ref{Lem-conti-square} and \ref{Lem_Integral_LLT}, we get the following result. 

\begin{lemma}\label{Lem_Integral_LLT-bb}
Let $\rho$ be a measurable and compactly supported function on $\bb R$. 
Then, there exists a constant $c = c(\rho) >0$ such that for 
any integrable function $H$ on $\bb R$ and any $n \geq 1$, 
\begin{align*}
\int_{\bb R}  \sup_{x, x' \in \bb S_+^{d-1}: x \neq x'} 
\frac{\bb E  \Big| \left( S_{n}^{x} - S_{n}^{x'}  \right) H *  \rho \left( y + S_n^x \right) \Big| }{ \mathbf d (x, x') }  dy
\leq   c  n^{ \frac{4 + \delta}{ 8 + 4 \delta } }  \int_{\bb R} |H(u)| du
\end{align*}
and 
\begin{align*}
\int_{\bb R}  \sup_{x, x' \in \bb S_+^{d-1}: x \neq x'} 
\frac{\bb E  \Big| \left( \min_{1 \leq j \leq n} S_{j}^{x} - \min_{1 \leq j \leq n} S_{j}^{x'}  \right) H *  \rho \left( y + S_n^x \right) \Big| }{ \mathbf d (x, x') }  dy
\leq   c  n^{ \frac{4 + \delta}{ 8 + 4 \delta } }  \int_{\bb R} |H(u)| du 
\end{align*}
where $\delta >0$ is from condition \ref{CondiMoment}. 
\end{lemma}

\begin{proof}
The assertion follows from Lemmas \ref{Lem-conti-square} and \ref{Lem_Integral_LLT}, together with the Cauchy-Schwarz inequality. 
\end{proof}

Now we introduce some notations which will be used in the sequel. 
For $\ee  >0 $, we define the function $\chi_{\ee}: \bb R \to \bb R_+$ as follows: 
\begin{align}\label{Def_chiee}
\chi_{\ee} (t) = 0  \  \mbox{for} \ t \leq -\ee,  
\   \chi_{\ee} (t) = \frac{t+\ee}{\ee}   \  \mbox{for} \ t \in (-\ee,0),
\   \chi_{\ee} (t) = 1  \  \mbox{for} \  t  \geq 0.  
\end{align}
Denote $\overline\chi_{\ee}(t) = 1 - \chi_{\ee}(t)$ and note that
\begin{align} \label{bounds-reversedindicators-001} 
\chi_{\ee}  \left( t-\ee \right) \leq  \mathds 1_{(0,\infty)} \left( t \right) \leq \chi_{\ee}  \left( t\right), 
\quad
\overline\chi_{\ee}  \left( t \right) \leq \mathds 1_{(-\infty,0]} \left( t \right) \leq \overline\chi_{\ee}  \left( t-\ee \right).
\end{align}
For any $H \in \mathscr H$ and any compactly supported measurable function $\rho$ on $\bb R$,  we denote
\begin{align}\label{def-H-star-rho}
H * \rho (x, y) = \int_{\bb R} H (x, y -v) \rho(v) dv,  
\quad  x \in \bb S_+^{d-1},  \  y \in \bb R. 
\end{align}

\begin{lemma}\label{Lem_Inequality_Aoverline}
Assume \ref{Condi-AP} and \ref{CondiMoment}. 
Let $\rho$ be a smooth compactly supported function on $\bb R$. 
Then, for any $H \in \mathscr H$, $m \geq 1$ and $\ee >0$,  
the function $L_{m, \ee}$ defined on $\bb S_+^{d-1} \times \bb R$ by 
\begin{align*}
L_{m, \ee}(x, y) :=   \mathbb{E} \left[ H *  \rho \left( X_m^x, y + S_{m}^x \right)   
  \overline\chi_{\ee}  \left( y -\ee + \min_{1 \leq  j \leq m}  S_j^x  \right)  \right],
\end{align*}
belongs to $\mathscr H$. 
Moreover, there exists a constant $c = c(\rho) > 0$ such that 
\begin{align}\label{inequ-convo-funcs}
\|  L_{m, \ee} \|_{\nu \otimes \Leb }  \leq   \int_{\bb R} |\rho(t)| dt  \| H \|_{\nu \otimes \Leb },  
\qquad 
 \| L_{m, \ee} \|_{ \mathscr H}  \leq  \frac{c}{\ee} m^{ \frac{4 + \delta}{ 8 + 4 \delta } }  \| H\|_{\mathscr H},
\end{align}
where $\delta >0$ is from condition \ref{CondiMoment}. 
\end{lemma}

\begin{proof}
As $0 \leq \overline\chi_{\ee} \leq 1$, we get
\begin{align*}
\|  L_{m, \ee} \|_{\nu \otimes \Leb }
=  \int_{ \bb S_+^{d-1} \times \bb R} |L_{m, \ee}(x, y) | \nu(dx) dy  
\leq  \int_{ \bb S_+^{d-1} \times \bb R} \mathbb{E} \left[ |H *  \rho| \left( X_m^x, y + S_{m}^x \right)  \right] \nu(dx) dy. 
\end{align*}
By Fubini's theorem, using the change of variable $y + S_m^x = y'$ 
and the stationarity of the invariant measure $\nu$, we have
\begin{align*}
\int_{ \bb S_+^{d-1} \times \bb R} \mathbb{E} \left[ |H *  \rho| \left( X_m^x, y + S_{m}^x \right)  \right] \nu(dx) dy
& =  \mathbb{E} \int_{\bb S_+^{d-1} \times \bb R}   |H *  \rho| \left( X_m^x, y + S_m^x \right)  dy \nu(dx)   \notag\\
& =  \mathbb{E} \int_{\bb S_+^{d-1} \times \bb R}    | H *  \rho |  \left( X_m^x, y' \right) dy' \nu(dx)    \notag\\
& =   \int_{\bb S_+^{d-1} \times \bb R}    | H *  \rho |  \left(x, y' \right) dy' \nu(dx)    \notag\\
& \leq  \int_{\bb R} |\rho(t)| dt  \  \| H \|_{\nu \otimes \Leb },
\end{align*}
where to get the last line we used the bound
\begin{align}\label{convolution-kappa-H}
\int_{\bb R}    | H *  \rho |  \left(x, y' \right) dy' 
&=   \int_{\bb R}  \left|  \int_{\bb R} H(x, y' - t) \rho(t) dt  \right|  dy' 
\leq  \int_{\bb R}  |\rho(t)|  \int_{\bb R}     \left| H(x, y'-t) \right| dy' dt  \notag \\ 
&=  \int_{\bb R} |\rho(t)| dt   \int_{\bb R}  \left| H(x, y'') \right| dy''. 
\end{align}
Putting together the previous bounds proves the first inequality in \eqref{inequ-convo-funcs}.  

Now we proceed to prove the second inequality in \eqref{inequ-convo-funcs}. 
Recall that 
\begin{align} \label{Holder-norm-L-m-ee-abc}
\|  L_{m, \ee} \|_{\mathscr H}
 =  \int_{\bb R}   \sup_{x \in \bb S_+^{d-1}}  | L_{m, \ee} \left(x, y \right)|  dy
 +  \int_{\bb R} \sup_{x, x' \in \bb S_+^{d-1}: x \neq x'}  
    \frac{| L_{m, \ee} \left(x, y \right)  - L_{m, \ee} \left(x', y \right)  |}{ \bf d(x, x')} dy. 
\end{align}
Set 
\begin{align*}
\bar H(y) = \sup_{x \in \bb S_+^{d-1}} |H(x, y)|, \quad  y \in \bb R. 
\end{align*}
Utilizing again the fact that $0 \leq \overline\chi_{\ee} \leq 1$, we have, for any $x \in \bb S_+^{d-1}$ and $y \in \bb R$, 
\begin{align*}
 | L_{m, \ee} \left(x, y \right) | 
& \leq  \bb E \left[ |\bar H * \rho| \left( y + S_{m}^x \right)  \right].  
\end{align*}
By integrating over $y$ and applying Lemma \ref{Lem_Integral_LLT}, 
we get that there exists a constant $c = c(\rho) > 0$ such that for any $m \geq 1$ and $\ee >0$, 
\begin{align}\label{Inequa-L001}      
\int_{\bb R}   \sup_{x \in \bb S_+^{d-1}}  | L_{m, \ee} \left(x, y \right)|  dy
 \leq   \int_{\bb R}  \sup_{x \in \bb S_+^{d-1}}  \bb E \left[ |\bar H *  \rho| \left( y + S_{m}^x \right)  \right]  dy 
 \leq  c m^{1/4} \int_{\bb R} \bar H(y) dy. 
\end{align}

Now we dominate the second term in \eqref{Holder-norm-L-m-ee-abc} and start with the following decomposition: 
\begin{align*}
| L_{m, \ee} \left(x, y \right)  - L_{m, \ee} \left(x', y \right)  | 
\leq   I_{m,1}(x,x',y) +  I_{m,2}(x,x',y) +  I_{m,3}(x,x',y), 
\end{align*}
where 
\begin{align*}
I_{m,1}(x, x', y) 
& =  \left| L_{m, \ee} \left(x, y \right) -  \mathbb{E} \left[ H *  \rho \left( X_m^x, y + S_{m}^{x'} \right)  
  \overline\chi_{\ee}  \left( y -\ee + \min_{1 \leq  j \leq m}  S_j^x  \right)  \right]  \right|,    \notag\\
I_{m,2}(x, x', y) & =  \left|   
    \mathbb{E}   \left[ H *  \rho \left( X_m^x, y + S_{m}^{x} \right)  -  H *  \rho \left( X_m^{x'}, t+S_{m}^{x} \right)    \right]
   \overline\chi_{\ee}  \left( y -\ee + \min_{1 \leq  j \leq m}  S_j^x  \right)    \right|,    \notag\\
I_{m,3}(x, x', y) & =  \left|   \bb E   
     H *  \rho \left( X_m^x, y + S_{m}^{x} \right) 
 \left[ \overline\chi_{\ee}  \left( y -\ee + \min_{1 \leq  j \leq m}  S_j^x \right) 
  -  \overline\chi_{\ee}  \left( y -\ee + \min_{1 \leq  j \leq m}  S_j^{x'} \right)  \right]    \right|.  
\end{align*}

\textit{Bound of $I_{m,1}(x, x', y)$.}
Set $\rho_2(y) = \sup_{s \in \bb R: |s| \leq c_0} |\rho'(y + s)|$ for $y \in \bb R$. 
Since $\rho'$ is supported on $[- c_0, c_0]$, 
the function $\rho_2$ is supported on $[- 2c_0, 2c_0]$. 
Hence, for any $x \in \bb S_+^{d-1}$ and $y, y' \in \bb R$, 
\begin{align*}
| H * \rho(x, y) - H * \rho(x, y') | \leq  |y - y'| \bar H * \rho_2(y). 
\end{align*}
It follows that for $x, x' \in \bb S_+^{d-1}$ and  $y \in \bb R$,  
\begin{align*}
I_{m,1}(x, x', y) 
 & \leq  \bigg| \mathbb{E} \Big[ H *  \rho \Big( X_m^x, y + S_{m}^{x} \Big)   \Big]
                    -  \mathbb{E} \left[ H *  \rho \left( X_m^x, y + S_{m}^{x'} \right)   \right]  \bigg|  \notag\\
 & \leq   \bb E \left[  \left| S_{m}^{x} - S_{m}^{x'}  \right|  \bar H * \rho_2 \left( y + S_{m}^x  \right)  \right].
\end{align*}
By Lemma \ref{Lem_Integral_LLT-bb}, there exists a constant $c = c(\rho_2) >0$ such that for any $m \geq 1$, 
\begin{align}\label{Bound_I1zzt}
\int_{\bb R}  \sup_{x, x' \in \bb S_+^{d-1}: x \neq x'} 
\frac{I_{m,1}(x, x', y) }{ \mathbf d (x, x') }  dy
\leq   c  m^{ \frac{4 + \delta}{ 8 + 4 \delta } }  \int_{\bb R} |\bar H(u)| du,  
\end{align}
where $\delta >0$ is from condition \ref{CondiMoment}.

\textit{Bound of $I_{m,2}(x, x', y)$.}
Since $H \in \mathscr H$, the function $L(y) = \sup_{x, x' \in \bb S_+^{d-1}: x \neq x'}  \frac{|H(x, y) - H(x', y)|}{ \bf d(x, x') }$
is integrable on $\bb R$. Hence, for $x, x' \in \bb S_+^{d-1}$, $y \in \bb R$ and $m \geq 1$, 
\begin{align*}
I_{m,2}(x, x', y) 
& \leq   \bb{E}   \left| H *  \rho \left( X_m^x, y + S_{m}^{x} \right)  -  H *  \rho \left( X_m^{x'}, y + S_{m}^{x} \right)  \right|  \notag\\ 
 & \leq   \bb E \left[  \bf d(X_m^x, X_m^{x'})  L * \rho \left(  y + S_{m}^{x} \right)  \right]    \notag\\
 & \leq    \bf d(x, x')  \bb E \left[ L * \rho \left(  y + S_{m}^{x} \right)  \right], 
\end{align*}
where in the last inequality we used \eqref{inequality-contraction}. 
Applying Lemma \ref{Lem_Integral_LLT}, we obtain that there exists a constant $c = c(\rho) >0$ such that for any $m \geq 1$, 
\begin{align}\label{Bound_I2zzt}
\int_{\bb R}  \sup_{x, x' \in \bb S_+^{d-1}: x \neq x'}  \frac{I_{m,2}(x, x', y) }{ \mathbf d (x, x') }  dy 
& \leq  c m^{1/4}  \int_{\bb R} |L(y)| dy \notag\\
& = c m^{1/4}  \int_{\bb R}  \sup_{x, x' \in \bb S_+^{d-1}: x \neq x'}  \frac{|H(x, y) - H(x', y)|}{ \bf d(x, x') } dy. 
\end{align}

\textit{Bound of $I_{m,3}(x, x', y)$.}
Since $\overline\chi_{\ee}$ is $1/\ee$-Lipschitz continuous on $\bb R$, we have
\begin{align}\label{bound-I3-xxy-abc}
 I_{m,3}(x, x', y)  
 & \leq  \frac{1}{\ee} \bb E    \left[
     H *  \rho \left( X_m^x, y + S_{m}^{x} \right) 
  \left| \min_{1 \leq  j \leq m}  S_j^x  - \min_{1 \leq  j \leq m}  S_j^{x'} \right|   \right]  \notag\\
  & \leq \frac{1}{\ee} \bb E    \left[
     \bar H *  \rho \left( y + S_{m}^{x} \right) 
  \left| \min_{1 \leq  j \leq m}  S_j^x  - \min_{1 \leq  j \leq m}  S_j^{x'} \right|   \right]. 
 \end{align}
 By Lemma \ref{Lem_Integral_LLT-bb}, there exists a constant $c = c(\rho) >0$ such that for any $m \geq 1$, 
\begin{align}\label{Bound_I3zzt}
\int_{\bb R}  \sup_{x, x' \in \bb S_+^{d-1}: x \neq x'}  \frac{I_{m,3}(x, x', y) }{ \mathbf d (x, x') }  dy 
 \leq  \frac{c}{\ee} m^{ \frac{4 + \delta}{ 8 + 4 \delta } }  \int_{\bb R}  \sup_{x, x' \in \bb S_+^{d-1}: x \neq x'}  \frac{|H(x, y) - H(x', y)|}{ \bf d(x, x') } dy. 
\end{align}
Putting together \eqref{Bound_I1zzt}, \eqref{Bound_I2zzt} and \eqref{Bound_I3zzt},
and integrating over $y \in \bb R$, yields the second inequality in \eqref{inequ-convo-funcs}.  
\end{proof}


\section{Proof of Theorems \ref{Thm-nodrift-ysmall} and \ref{Thm-nodrift-ysmall-target}} \label{sec-proof Th CondLLT Version1}

\subsection{A non-asymptotic conditioned local limit theorem}

In this section we state the following effective upper and lower bounds in the 
 conditioned local limit theorem of Caravenna-type for products of positive random matrices, 
 from which we shall deduce Theorems \ref{Thm-nodrift-ysmall} and \ref{Thm-nodrift-ysmall-target}. 
Our result is formulated with a target function $F$ on the pair $(X_n^x, S_n^x)$, 
which will also be applied to obtain Theorem \ref{Thm-CLLT-small-x}.

\begin{theorem} \label{t-B 002}
Assume \ref{Condi-AP}, \ref{CondiMoment},  \ref{Condi-Lya} and \ref{CondiNonarith}. 
Let $\delta \in (0, 1]$ be from \ref{CondiMoment}.  
Then, one can find constants $\beta_0, c >0$ and a sequence $(\ee_n)_{n \geq 1}$ of positive numbers satisfying
$\lim_{n\to\infty}\ee_n = 0$ such that the following holds. 
For any $\beta \in (0, \beta_0)$, there exists a constant $c_{\beta}>0$ such that
 for all $n \geq 1$, $x \in \bb S_+^{d-1}$ and $y \in \bb R$, 
\begin{enumerate}[label=\arabic*., leftmargin=*]
\item  For any  nonnegative function $F$ 
and any function $H \in \mathscr H$ satisfying $F \leq_{\ee_n} H$,  
$F(x, y)=0$ for $x \in \bb S_+^{d-1}$ and $y <0$, and $H(x, y)=0$ for $x \in \bb S_+^{d-1}$ and $y<-\ee_n$,      
\begin{align}\label{eqt-A 001}
& \bb E \Big[ F(X_n^x, y + S_n^x);  \tau_{x, y} > n  \Big]   \notag\\
&  \leq  (1 + c \ee_n) \frac{ V_n(x, y) }{\sigma^2 n}
 \int_{\bb S_+^{d-1} \times \bb R_+}    H \left(x, z \right)  
   \ell \left(\frac{y}{\sigma\sqrt{n}} ,\frac{z}{\sigma\sqrt{n}}\right) \nu(dx)  dz  \notag\\
& \quad  +  c_{\beta}  \frac{1 + V_n(x, y) }{ n^{1 + \beta} } \| H \|_{\nu \otimes \Leb} 
+  c_{\beta} \frac{1 + V_n(x, y) }{ n^{1 + \frac{\delta}{2} - \beta} }  \| H \|_{\mathscr H}.  
\end{align}
\item For any nonnegative  function $F$ and nonnegative  functions $H, M \in \mathscr H$
satisfying $M \leq_{\ee_n} F \leq_{\ee_n} H$, 
$F(x, y)=0$ for $x \in \bb S_+^{d-1}$ and $y <0$,  $H(x, y)=0$ for $x \in \bb S_+^{d-1}$ and $y<- \ee_n$,   
and $M(x, y)=0$ for $x \in \bb S_+^{d-1}$ and $y< \ee_n$,   
\begin{align} \label{eqt-A 002}
&  \bb E \Big[ F(X_n^x, y + S_n^x);  \tau_{x, y} > n  \Big]   \notag\\
&  \geq   \frac{ V_n(x, y) }{\sigma^2 n}
 \int_{\bb S_+^{d-1} \times \bb R_+}    (M \left(x, z \right) - c \ee_n H \left(x, z \right) )
   \ell \left(\frac{y}{\sigma\sqrt{n}} ,\frac{z}{\sigma\sqrt{n}}\right) \nu(dx)  dz    \notag\\
&\quad  -  c_{\beta}  \frac{1 + V_n(x, y)}{n^{1 + \beta}}   \| H \|_{\nu \otimes \Leb }    
   -   c_{\beta} \frac{1+ V_n(x, y)}{ n^{1 +  \delta_1} } 
\left( \left\Vert  H  \right\Vert _{\mathscr H} +  \left\Vert  M  \right\Vert _{\mathscr H}  \right), 
\end{align}
where $\delta_1 = \frac{ 2 \delta^2 + 3 \delta - 4 }{4(2 + \delta)} - \beta$.
Moreover, under the Furstenberg-Kesten condition \ref{Condi-FK} (in place of \ref{Condi-AP}), 
one can take $\delta_1 = \frac{\delta}{2}$ in \eqref{eqt-A 002}. 
\end{enumerate}
\end{theorem}

In particular, taking $F = \mathds 1_{[z, z + \Delta]}$ in \eqref{eqt-A 001}
and using the fact that the function $\ell$ is bounded, we get the following corollary. 

\begin{corollary}\label{Cor-Caravenna-001}
Assume \ref{Condi-AP}, \ref{CondiMoment},  \ref{Condi-Lya} and \ref{CondiNonarith}. 
Then there exists a constant $c>0$ such that for any $n \geq 1$, $x \in \bb S_+^{d-1}$, $y \in \bb R$, 
$z \in \bb R$ and $\Delta \in \bb R_+$, 
\begin{align*}
\bb P \left( y + S_n^x \in  [z, z + \Delta],   \tau_{x, y} > n   \right)
\leq \frac{c}{n} \left( 1 + V_n(x, y) \right) \left( 1 + \Delta \right). 
\end{align*}
\end{corollary}

Let $\phi_{v}$ be the normal density of variance $v > 0$:  $\phi_{v} (t) = \frac{1}{\sqrt{v}} \phi \left( \frac{t}{\sqrt{v}} \right)$,
$t \in \bb R$. 
As in \eqref{Def-Levydens}, the scaled heat kernel $\psi_v$ with scale parameter $\sqrt{v}$ is defined as
\begin{align} \label{Rayleigh dens with scale v 001}
\psi_{v}(x, y) 
= \frac{1}{\sqrt{v}} \psi \left(\frac{x}{\sqrt{v}}, \frac{y}{\sqrt{v}}\right)
= \phi_{v}(x-y) - \phi_{v}(x+y), \quad x,y\in \bb R. 
\end{align}
We denote, for $v > 0$,  
\begin{align} \label{new-density-scale-v}
\ell_{v}(x, y) = \frac{\psi_{v}(x, y)}{H(x)},  \quad x, y \in \bb R. 
\end{align}
For $x, y \in \bb R$, let 
\begin{align} \label{def-ell-H-plus-minus}
\ell_{v}^+(x,y)= \ell_{v}(x,y) \mathds 1_{\{ y > 0 \}} 
\qquad  \mbox{and}  \qquad
\ell_{v}^-(x,y)= \ell_{v}(x,y) \mathds 1_{\{ y < 0 \}}. 
\end{align}
We define, for any $x, y \in \bb R$, 
\begin{align} 
& \phi_{v} * \ell_{1-v}^+(x, y) := \int_{0}^{\infty} \phi_{v} (y-z) \ell^+_{1-v}(x, z) dz,  \label{def-convo-ell-H-plus} \\
& \phi_{v} * \ell_{1-v}^-(x, y) := \int_{-\infty}^0  \phi_{v} (y-z) \ell^-_{1-v}(x, z) dz.  \label{def-convo-ell-H-minus}
\end{align}

The following lemmas are from \cite[Lemmas 5.3 and 5.10]{GX25}. 

\begin{lemma} \label{convol-phi-psi-001} 
For any $v \in (0,1)$ and $x, y \in \bb R$, we have
\begin{align} \label{ConvoNormalLevy02}
\phi_{v} * \ell^+_{1-v}(x,y) +  \phi_{v} * \ell^-_{1-v}(x, y)  =  \ell (x, y). 
\end{align} 
\end{lemma}

\begin{lemma} \label{Lipschiz  property for ell_H in y} 
There exists a constant $c > 0$ such that, for any $x, y, a \in \bb R$, 
 \begin{align} \label{ell-intergr-bound-002}
\left| \ell(x, y+a)  -   \ell(x, y) \right|  \leq c |a|. 
\end{align}
\end{lemma}

\subsection{Proof of the upper bound} 
In this section we prove the upper bound \eqref{eqt-A 001} of Theorem \ref{t-B 002}.
Set 
$m=[ n^{1- \beta}]$ and $k = n-m$, where $\beta \in (0,\frac{1}{2}]$ will be chosen to be sufficiently small later.  
For short, we denote for any $(x, y) \in \bb S_+^{d-1} \times \mathbb R$ and $n \geq 1$, 
\begin{align*}
I_n(x, y) : = \bb E \Big[ F(X_n^x, y + S_n^x);  \tau_{x, y} > n  \Big]. 
\end{align*}
By the Markov property of $(X_n^x, S_n^x)$, 
it holds that for any $(x, y) \in \bb S_+^{d-1} \times \mathbb R$ and $n \geq 1$, 
\begin{align} \label{JJJ-markov property}
I_n(x, y) 
& = \int_{\bb S_+^{d-1} \times \mathbb R_+}  I_m(x', y') \bb P \left(X_k^x \in dx', y + S_k^x \in dy',  \tau_{x, y} > k \right)   \notag\\
& \leq \int_{\bb S_+^{d-1} \times \mathbb R_+}  
\mathbb E  \Big[ F(X_m^{x'}, y' + S_{m}^{x'} ) \Big]  
   \bb P \left(X_k^x \in dx', y + S_k^x \in dy', \tau_{x, y} > k \right). 
\end{align}
Using the effective local limit theorem (\eqref{LLT-general001} of Theorem \ref{LLT-general}), 
with $\delta \in (0, 1]$ from \ref{CondiMoment}, 
one can find a constant $c > 0$ such that the following holds:
for any $\ee \in (0, \frac{1}{8})$, there exist constants $c_{\ee}, c_{\beta} >0$ such that for any $(x', y') \in \bb S_+^{d-1} \times \bb R_+$, 
\begin{align} \label{JJJJJ-1111-001}
 \mathbb E  \Big[ F \left( X_m^{x'}, y' + S_{m}^{x'} \right)   \Big]    
& \leq  (1 + c \ee) \int_{\bb S_+^{d-1} \times \bb R } 
  H(x, y)  \frac{1}{\sigma \sqrt{m}} \phi \left(\frac{y - y'}{\sigma \sqrt{m}} \right) \nu(dx) dy  \notag\\
 & \quad  +  \frac{c \ee}{n^{1-\beta}} \| H \|_{\nu \otimes \Leb}
     +   \frac{c_{\ee} c_{\beta}}{ n^{(1-\beta)(1 + \delta)/2} }  \| H \|_{\mathscr H}.   
\end{align}
Substituting \eqref{JJJJJ-1111-001} into \eqref{JJJ-markov property}, using \eqref{bound-tau-y-bis} 
and the fact that $V_k(x, y) \leq  V_n(x,y)$,  
we get that for any $(x, y) \in \bb S_+^{d-1} \times \bb R$,
\begin{align} \label{JJJ004}
I_n(x, y)    
  \leq  (1 + c\ee) J_n(x, y) +  c \ee \frac{ 1+ V_n(x, y) }{n^{\frac{3}{2} - \beta}} \| H \|_{\nu \otimes \Leb}  
  +  c_{\ee} c_{\beta} \frac{1 + V_n(x, y)}{  n^{ 1 + \frac{\delta}{2} - \frac{\beta(1+\delta)}{2} } }  \| H \|_{\mathscr H},  
\end{align}
where 
\begin{align} \label{JJJ006}
 J_n(x, y)   
: &=  \int_{\bb R_{+}}   \left[ \int_{\bb S_+^{d-1} \times \bb R } 
  H(x, y)  \frac{1}{\sigma \sqrt{m}} \phi \left(\frac{y - y'}{\sigma \sqrt{m}} \right) \nu(dx) dy \right]
    \bb P \left( y + S_k^x \in dy', \tau_{x, y} > k \right)   \notag\\
& =  \int_{\bb R_{+}}   T_n(y')  \bb P \left( \frac{y + S_k^x}{\sigma \sqrt{k}} \in dy', \tau_{x, y} > k \right).   
\end{align}
and 
\begin{align}\label{Def-Lmt}
T_n(y') : & = \int_{\bb S_+^{d-1} \times \bb R } 
  H(x, y)  \frac{1}{\sigma \sqrt{m}} \phi \left(\frac{y - \sigma \sqrt{k} y'}{\sigma \sqrt{m}} \right) \nu(dx) dy  \notag\\
 & = \int_{\bb S_+^{d-1} \times \bb R } 
  H \left(x, \sigma \sqrt{k} y \right)  \frac{1}{\sqrt{m/k}}  \phi \left(\frac{y - y'}{\sqrt{m/k}} \right)    \nu(dx)  dy. 
\end{align}
Since the function $y' \mapsto T_n(y')$ is differentiable on $\bb R$ and vanishes as $y' \to \infty$, 
using Fubini's theorem, we get that for any $(x, y) \in \bb S_+^{d-1} \times \bb R$,
\begin{align} \label{ApplCondLT-001}
J_{n}(x, y) 
& =  \bb E  \left[  T_n \left( \frac{y + S_k^x}{\sigma \sqrt{k}} \right);  \tau_{x, y} > k \right]   \notag\\
& =  \bb E  \left[ - \int_{ \frac{y + S_k^x}{\sigma \sqrt{k}} }^{\infty} T_n'(t) dt;  \tau_{x, y} > k \right]   \notag\\
& = - \int_{\bb R_+}  T_n'(t)  \bb P \left(  \frac{y + S_k^x}{\sigma \sqrt{k}} \leq  t, \tau_{x, y} > k\right) dt.
 \end{align}
By the conditioned integral limit theorem (Theorem \ref{Theor-probIntUN-002}), 
there exists $\eta>0$ such that for any $\eta \in (0, \eta_0)$, 
$k \geq 1$, $x \in \bb S_+^{d-1}$, $y \in \bb R$ and $t \in \bb R_+$, 
\begin{align*}
\left|\bb P \left( \frac{y + S_k^x}{\sigma\sqrt{k}} \leq t, \tau_{x, y} > k\right)
 -  \frac{V_k(x,y)}{\sigma\sqrt{k}}   \int_{0}^{t}  \ell \left(\frac{y}{\sigma \sqrt{k}}, u   \right)  du  \right| 
 &  \leq  c_{\eta}  \frac{  1 + V_k(x, y) }{k^{1/2+\eta}}. 
\end{align*} 
Together with \eqref{ApplCondLT-001}, this implies that 
for any $n \geq 1$, $x \in \bb S_+^{d-1}$ and $y \in \bb R$, 
\begin{align} \label{ApplCondLT-small-y}
 \left| J_{n}(x, y)  +   \frac{V_k(x,y)}{\sigma\sqrt{k}}
   \int_{\bb R_+}  L_n'(t)  \left[ \int_{0}^{t}  \ell \left(\frac{y}{\sigma \sqrt{k}}, u   \right)  du \right]  dt \right|   
 \leq  c_{\eta}  \frac{  1 + V_k(x, y) }{ k^{1/2+\eta} } 
  \int_{\bb R_+}  |L_n'(t)|  dt. 
\end{align}
In view of \eqref{Def-Lmt}, using a change of variables $\frac{y'}{ \sqrt{m/k} } = y$ and $\frac{t}{ \sqrt{m/k} } = s$, 
and the fact that $\int_{\bb R} |\phi'(t)| dt < 1$, we get
\begin{align}\label{Inte-Lmt-sma-y}
 \int_{\bb R_+}  |T_n'(t)|  dt  
& \leq  \int_{\bb R_+} 
 \left[  \int_{\bb S_+^{d-1} \times \bb R } 
  H \left(x, \sigma \sqrt{k} y' \right)   
   \left| \phi' \left(\frac{y' - t}{\sqrt{m/k}} \right) \right|  \nu(dx)  \frac{dy'}{ \sqrt{m/k} }  \right]  \frac{dt}{\sqrt{m/k}}   \notag\\
& =  \int_{\bb R_+} 
 \left[  \int_{\bb S_+^{d-1} \times \bb R } 
  H \left(x, \sigma \sqrt{m} y \right)   
   \left| \phi' \left( y - s \right) \right|  \nu(dx)  dy  \right]  ds   \notag\\
& <  \int_{\bb S_+^{d-1} \times \bb R } 
  H \left(x, \sigma \sqrt{m} y \right)  \nu(dx)  dy  \notag\\
&  = \frac{1}{\sigma \sqrt{m}}  \| H \|_{\nu \otimes \Leb}. 
\end{align}
Since $\lim_{t \to \infty} T_n(t) = 0$ and $\ell ( \frac{y}{\sigma \sqrt{k}}, 0) = 0$, using integration by parts, we have
\begin{align*}
A: = - \int_{\bb R_+}  L_n'(t)  \left[ \int_{0}^{t}  \ell \left(\frac{y}{\sigma \sqrt{k}}, u   \right)  du \right]  dt
   = \int_{\bb R_+}  L_n(t)  \ell \left(\frac{y}{\sigma \sqrt{k}}, t  \right)  dt. 
\end{align*}
By the definition of $T_n$ (cf.\  \eqref{Def-Lmt}), a change of variables $y' = \sqrt{ \frac{n}{k} } s$ and $t' = \sqrt{ \frac{n}{k} } t$, 
it follows that
\begin{align*}
A & =   \int_{\bb R _{+}}   \left[  \int_{\bb S_+^{d-1} \times \bb R } H \left(x, \sigma \sqrt{k} y' \right)  
    \frac{1}{\sqrt{m/k}}  \phi \left(\frac{y' - t'}{\sqrt{m/k}} \right)   \nu(dx)  dy' \right]  \ell \left(\frac{y}{\sigma \sqrt{k}}, t'  \right) dt'  \notag\\
& =  \int_{\bb R _{+}}   \left[  \int_{\bb S_+^{d-1} \times \bb R } H \left(x, \sigma \sqrt{n} s \right)  
    \frac{1}{\sqrt{m/k}}  \phi \left(\frac{s - t}{\sqrt{m/n}} \right)  \nu(dx)  \frac{ds}{\sqrt{k/n}}   \right]  
      \ell \left(\frac{y}{\sigma \sqrt{k}}, \frac{t}{\sqrt{k/n}} \right) \frac{dt}{\sqrt{k/n}}  \notag\\
& =  \int_{\bb R _{+}}   \left[  \int_{\bb S_+^{d-1} \times \bb R } H \left(x, \sigma \sqrt{n} s \right)  
     \phi_{\delta_n} \left(s - t \right)  \nu(dx)  ds   \right]  
      \ell \left(\frac{y}{\sigma \sqrt{k}}, \frac{t}{\sqrt{k/n}} \right) \frac{dt}{\sqrt{k/n}}, 
\end{align*}
where $\delta_n = \frac{m}{n}$. 
Since $1 - \delta_n = \frac{k}{n}$ and $\frac{V_k(x,y)}{ \sigma \sqrt{k}} \frac{1}{H( \frac{y}{\sigma \sqrt{k}})} 
= \frac{V_n(x,y)}{ \sigma \sqrt{n}} \frac{1}{H( \frac{y}{\sigma \sqrt{n}})}$,
we get 
\begin{align}\label{heat-kernel-esti-002}
B : & = \frac{V_k(x,y) }{\sigma\sqrt{k}} A   \notag\\
& =   \frac{V_k(x,y) }{\sigma\sqrt{k}} \int_{\bb R _{+}}  
 \left[  \int_{\bb S_+^{d-1} \times \bb R } H \left(x, \sigma \sqrt{n} s \right)  
     \phi_{\delta_n} \left(s - t \right)  \nu(dx)  ds   \right]  
     \frac{\psi_{1-\delta_n}\left(\frac{y}{\sigma\sqrt{n}}, t \right)}{H\left( \frac{y}{\sigma \sqrt{k}} \right)} dt  \notag\\
& =   \frac{V_n(x,y) }{\sigma\sqrt{n}} 
 \int_{\bb R _{+}}   \left[  \int_{\bb S_+^{d-1} \times \bb R } H \left(x, \sigma \sqrt{n} s \right)  
     \phi_{\delta_n} \left(s - t \right)  \nu(dx)  ds   \right]  
     \ell_{1-\delta_n}\left(\frac{y}{\sigma\sqrt{n}}, t \right) dt, 
\end{align}
where $\ell_{1-\delta_n}$ is defined by \eqref{new-density-scale-v}. 
Using Fubini's theorem and the definition of $\phi_{\delta_n}$ and $\ell^+_{1-\delta_n}$ (cf.\ \eqref{def-ell-H-plus-minus}),
we get
\begin{align}\label{equa-term-A-01}
B  & = \frac{V_n(x,y) }{\sigma\sqrt{n}}
 \int_{\bb S_+^{d-1} \times \bb R } H \left(x, \sigma \sqrt{n} s \right)   
   \phi_{\delta_n} * \ell^+_{1-\delta_n} \left(\frac{y}{\sigma \sqrt{n}}, s \right) \nu(dx)  ds   \notag\\
& =   \frac{V_n(x,y) }{\sigma^2 n}  \int_{\bb S_+^{d-1}}  \int_{-\ee}^{\infty}  H \left(x, s \right)   
   \phi_{\delta_n} * \ell^+_{1-\delta_n} \left( \frac{y}{\sigma \sqrt{n}},  \frac{s}{\sigma \sqrt{n}} \right) \nu(dx)  ds, 
\end{align}
where in the last line we used the fact that $H(x, s)=0$ for any $x \in \bb S_+^{d-1}$ and $s \leq - \ee$. 
Applying Lemma \ref{convol-phi-psi-001} with $v = \delta_n$, and the fact that $1 - v = \frac{k}{n}$,
we get that for any $s \geq - \ee$, 
\begin{align*}
& \phi_{\delta_n}*\ell^+_{1-\delta_n}\left(\frac{y}{\sigma\sqrt{n}}, \frac{s}{\sigma\sqrt{n}}\right)    \notag\\
& = \ell \left(\frac{y}{\sigma\sqrt{n}}, \frac{s}{\sigma\sqrt{n}}\right)  
 -  \phi_{\delta_n}*\ell^-_{1-\delta_n}\left(\frac{y}{\sigma\sqrt{n}}, \frac{s}{\sigma\sqrt{n}}\right)  \notag\\
 & = \ell \left(\frac{y}{\sigma\sqrt{n}}, \frac{s}{\sigma\sqrt{n}}\right)  
 -  \int_{\bb R}  \phi_{\delta_n} (z) \ell^-_{1-\delta_n} \left( \frac{y}{\sigma \sqrt{n}}, \frac{s}{\sigma\sqrt{n}} - z \right) dz,
\end{align*}
where in the last line we used \eqref{def-convo-ell-H-minus}. 
Using Lemma \ref{Lipschiz  property for ell_H in y}
and the fact that $\ell^-_{1 - \delta_n} ( \frac{y}{\sigma \sqrt{n}}, \frac{s}{\sigma\sqrt{n}} ) = 0$
for any $y \in \bb R$ and $s \in \bb R_+$, there exists $c>0$ such that for any $s, y, z \in \bb R$, 
\begin{align*} 
\left| \ell^-_{1-\delta_n} \left( \frac{y}{\sigma \sqrt{n}}, \frac{s}{\sigma\sqrt{n}} - z \right)  \right|
= \left| \ell^-_{1-\delta_n} \left( \frac{y}{\sigma \sqrt{n}}, \frac{s}{\sigma\sqrt{n}} - z \right)
- \ell^-_{1-\delta_n} \left( \frac{y}{\sigma \sqrt{n}}, \frac{s}{\sigma\sqrt{n}} \right) \right| \leq c |z|. 
\end{align*}
Since $\int_{\bb R} \phi_{\delta_n} (z) |z| dz \leq  \sqrt{\delta_n} =\sqrt{ \frac{m}{n}} \leq c_{\beta} n^{-\beta/2}$, we get
\begin{align}\label{bound-ell-psi-a}
& \int_{\bb S_+^{d-1}}  \int_{0}^{\infty}  H \left(x, s \right)   
   \phi_{\delta_n} * \ell^+_{1-\delta_n} \left( \frac{y}{\sigma \sqrt{n}},  \frac{s}{\sigma \sqrt{n}} \right) \nu(dx)  ds  \notag\\
& \leq  \int_{\bb S_+^{d-1}}  \int_{0}^{\infty}  H \left(x, s \right)     \ell \left(\frac{y}{\sigma\sqrt{n}} ,\frac{s}{\sigma\sqrt{n}}\right) \nu(dx)  ds
  + \frac{c_{\beta}}{n^{\beta/2}} \| H \|_{\nu \otimes \Leb}. 
\end{align}
Using again Lemma \ref{Lipschiz  property for ell_H in y}, we have 
$| \ell^+_{1-\delta_n}  ( \frac{y}{\sigma \sqrt{n}}, \frac{s}{\sigma\sqrt{n}} - z )| \leq c ( \frac{1}{\sqrt{n}} + |z| )$
for any $y, z \in \bb R$ and $s \in [- \ee, 0]$ with $\ee <1/2$, so that
\begin{align*}
\left| \phi_{\delta_n}*\ell^+_{1-\delta_n}\left(\frac{y}{\sigma\sqrt{n}}, \frac{s}{\sigma\sqrt{n}}\right) \right|
& = \left|  \int_{\bb R}  \phi_{\delta_n} (z) \ell^+_{1-\delta_n} \left( \frac{y}{\sigma \sqrt{n}}, \frac{s}{\sigma\sqrt{n}} - z \right) dz \right| \notag\\
& \leq  c  \int_{\bb R}  \phi_{\delta_n} (z) \left(  \frac{1}{\sqrt{n}}  + |z| \right) dz 
 \leq    \frac{c_{\beta}}{n^{\beta/2}}, 
\end{align*}
where we used $\beta \in (0,\frac{1}{2}]$. 
Hence 
\begin{align}\label{bound-ell-psi-b} 
 \left| \int_{\bb S_+^{d-1}} \int_{-\ee}^0  H \left(x, s \right)  
 \phi_{\delta_n}*\ell^+_{1-\delta_n}\left(\frac{y}{\sigma\sqrt{n}} ,\frac{s}{\sigma\sqrt{n}}\right) \nu(dx) ds  \right| 
 \leq   \frac{c}{n^{\beta/2}}   \int_{\bb S_+^{d-1}} \int_{- \ee}^0  H \left(x, s \right)  \nu(dx) ds. 
\end{align}
Substituting \eqref{bound-ell-psi-a} and \eqref{bound-ell-psi-b} into \eqref{equa-term-A-01}, we get
\begin{align}\label{heat-kernel-esti-004}
B   \leq  \frac{V_n(x,y) }{\sigma^2 n}
 \int_{\bb S_+^{d-1}}  \int_{0}^{\infty}  H \left(x, s \right)     \ell \left(\frac{y}{\sigma\sqrt{n}} ,\frac{s}{\sigma\sqrt{n}}\right) \nu(dx)  ds  
  +  c_{\beta}  \frac{1 + V_n(x,y)}{ n^{1 + \beta/2} } \| H \|_{\nu \otimes \Leb}.  
\end{align}
Substituting \eqref{Inte-Lmt-sma-y} and \eqref{heat-kernel-esti-004} into \eqref{ApplCondLT-small-y}, 
and choosing $\beta<\eta$, we derive that
\begin{align}\label{Bound-Jn-x-abc}
 \left| J_{n}(x, y) -  \frac{V_n(x,y) }{\sigma^2 n}
 \int_{\bb S_+^{d-1}}  \int_{0}^{\infty}  H \left(x, s \right)  \ell \left(\frac{y}{\sigma\sqrt{n}} ,\frac{s}{\sigma\sqrt{n}}\right) \nu(dx)  ds \right|   
 \leq c_{\beta}  \frac{1 + V_n(x,y) }{ n^{1 + \beta/2} } \| H \|_{\nu \otimes \Leb}.  
\end{align}
Combining this with \eqref{JJJ004}, 
we obtain 
\begin{align}\label{last-inequality-001}
  I_n(x, y) 
& \leq   (1 + c\ee) \frac{V_n(x,y) }{\sigma^2 n}
 \int_{\bb S_+^{d-1}}  \int_{0}^{\infty}  H \left(x, z \right)     \ell \left(\frac{y}{\sigma\sqrt{n}} ,\frac{z}{\sigma\sqrt{n}}\right) \nu(dx)  dz   \notag\\
& \quad   
+   c_{\beta} \frac{1 + V_n(x,y) }{ n^{1 + \beta/2} } \| H \|_{\nu \otimes \Leb}  \notag\\
& \quad +    c \ee \frac{ 1+ V_n(x, y) }{n^{\frac{3}{2} - \beta}} \| H \|_{\nu \otimes \Leb}  
  +  c_{\ee}  c_{\beta} \frac{1 + V_n(x, y)}{  n^{ 1 + \frac{\delta}{2} - \frac{\beta(1+\delta)}{2} } }  \| H \|_{\mathscr H}. 
\end{align}
Since $\beta \in (0, \eta)$ with $\eta$ small,
we have $\frac{1}{n^{\frac{3}{2} - \beta}} \leq \frac{1 }{ n^{1 + \beta/2} }$. 
Now we choose a sequence $(\ee_n)_{n \geq 1}$ of positive numbers satisfying $\lim_{n \to \infty} \ee_n = 0$ and
$c_{\ee_n} n^{\frac{\beta (1 + \delta)}{2}} \leq n^{2\beta}$, 
so that $\frac{1}{  n^{ 1 + \frac{\delta}{2} - \frac{\beta(1+\delta)}{2} } } \leq  \frac{1}{n^{1 + \frac{\delta}{2} - 2\beta}}$. 
From \eqref{last-inequality-001}, we deduce the upper bound \eqref{eqt-A 001}.

\subsection{Proof of the lower bound}
The aim of this section is to establish the lower bound \eqref{eqt-A 002} of Theorem \ref{t-B 002},
whose proof turns out to be more intricate and delicate than that of the upper bound \eqref{eqt-A 001}. 
Let us retain the notation used in the proof of \eqref{eqt-A 001}. 
In particular,  $m=[ n^{1- \beta}]$ and $k = n-m$, where $\beta \in (0,\frac{1}{2}]$. 
As in \eqref{JJJ-markov property}, 
the Markov property of $(X_n, S_n)$ implies that for any $(x, y) \in \bb S_+^{d-1} \times \mathbb R$ and $n \geq 1$, 
\begin{align} \label{JJJ-markov property-Low}
I_n(x, y) :& = \bb E \Big[ F(X_n^x, y + S_n^x);  \tau_{x, y} > n  \Big]    
 = : I_{n,1}(x, y) - I_{n,2}(x, y),  
\end{align}
where  
\begin{align*}
I_{n,1}(x, y) & =  \int_{\bb S_+^{d-1} \times \mathbb R_+}  
\mathbb E  \left[ F \left( X_m^{x'}, y' + S_{m}^{x'} \right)  \right]  
\bb P \left(X_k^x \in dx', y + S_k^x \in dy', \tau_{x, y} > k \right),   \notag \\
I_{n,2}(x, y) & =   \int_{\bb S_+^{d-1} \times \mathbb R_+}  
 \bb{E} \left[ F \left( X_m^{x'}, y'+S_{m}^{x'} \right);  \tau_{x', y'} \leq m  \right]   
\bb P \left(X_k^x \in dx', y + S_k^x \in dy', \tau_{x, y} > k \right).   
\end{align*}

\textit{Lower bound of $I_{n,1}(x, y)$.}
Using the effective local limit theorem \eqref{LLT-general002}, with $\delta \in (0, 1]$ from \ref{CondiMoment}, 
we derive that there exists a constant $c > 0$ such that the following holds:
for any $\ee \in (0, \frac{1}{8})$, 
there exist constants $c_{\ee}, c_{\beta} >0$ such that for any $(x', y') \in \bb S_+^{d-1} \times \bb R$, $m \geq 1$, 
any nonnegative  function $F$ and nonnegative  functions $H, M \in \mathscr H$
satisfying $M \leq_{\ee} F \leq_{\ee} H$, 
\begin{align}\label{Pf_SmallStarting_Firstthm}
 \mathbb E  \left[ F \left( X_m^{x'}, y' + S_{m}^{x'} \right)   \right]
& \geq   \frac{1}{ \sigma \sqrt{m}} \int_{\bb S_+^{d-1} \times \bb R } 
    (M \left( x, y \right) - c \ee H(x, y) )  \phi \left(\frac{y - y'}{\sigma \sqrt{m}} \right)   \nu(dx)  dy     \notag\\
& \quad    -  \frac{ \ee c_{\beta}  }{n^{1-\beta}}  \| H \|_{\nu \otimes \Leb}   -  
  \frac{ c_{\ee} c_{\beta} }{ n^{(1-\beta))(1 + \delta)/2} }  \left(  \| H \|_{\mathscr H}  +  \| M \|_{\mathscr H} \right).   
\end{align}
Proceeding in the same way as in \eqref{Bound-Jn-x-abc} on the estimate of $J_n(x, y)$ defined by \eqref{JJJ006}, 
choosing $\beta \in (0, \eta)$ with $\eta$  from Theorem \ref{Theor-probIntUN-002}, 
one has, uniformly in $x \in \bb S_+^{d-1}$ and $y \in \bb R$,  
\begin{align*}
& \frac{1}{\sigma \sqrt{m} } \int_{\mathbb R_+}  \left[ \int_{\bb S_+^{d-1} \times \bb R } 
   (M(x', u) - c \ee H(x', u) )  \phi \left(\frac{u - y'}{\sigma \sqrt{m}} \right)   \nu(dx')  du  \right]
   \bb P \left( y + S_k^x  \in dy', \tau_{x, y} > k \right)    \notag\\
& \geq  \frac{V_n(x,y) }{\sigma^2 n}
 \int_{\bb S_+^{d-1}}  \int_{0}^{\infty}  (M \left(x, z \right) - c \ee H \left(x, z \right) ) 
     \ell \left(\frac{y}{\sigma\sqrt{n}} ,\frac{z}{\sigma\sqrt{n}}\right) \nu(dx)  dz    \notag\\
& \quad   -  c_{\beta}  \frac{  1 + V_n(x,y)  }{n^{1 + \beta/2}}  \| M \|_{\nu \otimes \Leb}. 
\end{align*}
Therefore, using \eqref{bound-tau-y-bis}, 
we deduce that uniformly in $x \in \bb S_+^{d-1}$ and $y \in \bb R$, 
\begin{align}\label{Lower_In_lll}
 I_{n,1}(x,y)  
 & \geq    \frac{V_n(x,y) }{\sigma^2 n}
 \int_{\bb S_+^{d-1}}  \int_{0}^{\infty}  (M \left(x, z \right) - c \ee H \left(x, z \right) ) 
      \ell \left(\frac{y}{\sigma\sqrt{n}} ,\frac{z}{\sigma\sqrt{n}}\right) \nu(dx)  dz    \notag\\
& \quad   -  c_{\beta}  \frac{  1 + V_n(x,y)  }{n^{1 + \beta/2}}  \| H \|_{\nu \otimes \Leb}    
 -  c_{\ee} c_{\beta} \frac{1+ V_n(x,y) }{ n^{ 1 + \frac{\delta}{2} - \frac{\beta(1+\delta)}{2} } }
      \left(  \| H \|_{\mathscr H}  +  \| M \|_{\mathscr H} \right). 
\end{align}

\textit{Upper bound of $I_{n,2}(x, y)$.}
We proceed to give an upper bound for $I_{n,2}(x, y)$ defined in \eqref{JJJ-markov property-Low}, which can be rewritten as
\begin{align} 
I_{n,2}(x, y) = \int_{\bb S_+^{d-1} \times \mathbb R_+}  F_m(x', y')  \bb P \left( X_k^x \in dx', y + S_k^x \in dy', \tau_{x, y} > k \right),  \label{K2-b01c-001}
\end{align}
where, for $x' \in \bb S_+^{d-1}$ and $y' \in \bb R$, 
\begin{align} \label{Bytheduality-001}
F_m(x', y') : =  \mathds 1_{\{y' \geq 0\}}  \bb{E} \left[ F \left( X_m^{x'}, y'+S_{m}^{x'} \right);  \tau_{x', y'} \leq m  \right] . 
\end{align}
The function $y' \mapsto F_m(x', y')$ is integrable on $\bb R$ since  $y' \mapsto F(x', y')$ is integrable on $\bb R$.
Let $\kappa$ be a smooth and non-negative  function compactly supported in $[-1,1]$ such that
$\int_{-1}^1 \kappa(t)dt=1$ 
and set $\kappa_{\ee}(t) = \frac{1}{\ee} \kappa(\frac{t}{\ee})$ for $t \in \bb R$ and $\ee>0$. 
Denote, for $\ee>0$,  $x' \in \bb S_+^{d-1}$ and $y' \in \bb R$,  
\begin{align} \label{DefM88}
H_m(x', y') :=  \mathds 1_{\{y' + \frac{\ee}{2} \geq 0\}}
\bb{E} \left[ \widetilde{H} * \kappa_{\ee} \left( X_m^{x'}, y'+S_{m}^{x'} \right) 
\overline\chi_{\ee/2}  \left( y' -\ee + \min_{1 \leq  j \leq m}  S_j^{x'}  \right)  \right],     
\end{align}
where $F \leq_{\ee/2} \widetilde{H} \leq_{\ee/2} H$, 
$\chi_{\ee/2}$ is the same as in \eqref{Def_chiee} and $\overline\chi_{\ee/2} = 1 - \chi_{\ee/2}$. 
Then $F_m \leq_{\ee/2}  H_m$ since $F \leq_{\ee/2} \widetilde{H}$. 
By \eqref{inequ-convo-funcs} of Lemma \ref{Lem_Inequality_Aoverline}, we get
\begin{align}\label{bound-Li-norm-M}
\| H_m \|_{\nu \otimes \Leb} \leq \int_{\bb R} \kappa_{\ee} (t) dt \left\|  H  \right\|_{\nu \otimes \Leb} 
\leq  \left\|  H  \right\|_{\nu \otimes \Leb}. 
\end{align}
Therefore, applying the upper bound \eqref{last-inequality-001} to the function $F_m$, 
and using the fact that $V_k(x,y) \leq V_n(x,y)$, 
we obtain that, there exists $\beta_0 >0$ such that for any $\beta \in (0, \beta_0)$, $x\in \bb R$ and $y \in \bb R$, 
\begin{align}\label{eqt-A 001_Lower}
I_{n,2}(x, y) 
&  \leq  c \frac{V_n(x,y) }{n}
 \int_{\bb S_+^{d-1}}  \int_{\bb R_+}  H_m \left(x, z \right)     \ell \left(\frac{y}{\sigma\sqrt{k}} ,\frac{z}{\sigma\sqrt{k}}\right)  dz \, \nu(dx)  \notag\\
& \quad  +   c_{\beta}  \frac{1 + V_n(x,y) }{ n^{1 + \beta} } \| H \|_{\nu \otimes \Leb} 
+  c_{\beta} \frac{1 + V_n(x,y) }{ n^{1 + \frac{\delta}{2} - \beta }  }  \| H_m \|_{\mathscr H}.   
\end{align}
For the first term on the right-hand side of \eqref{eqt-A 001_Lower}, we write 
\begin{align}\label{decom-integral-gm-ell}
 \int_{\bb R_+}  H_m \left(x, z \right)     \ell \left(\frac{y}{\sigma\sqrt{k}} ,\frac{z}{\sigma\sqrt{k}}\right)  dz
& = \int_{0}^{n^{1/2 - 3 \beta/8}}   H_m \left(x, z \right)   \ell \left(\frac{y}{\sigma\sqrt{k}} ,\frac{z}{\sigma\sqrt{k}}\right)  dz  \notag\\
& \quad + \int_{ n^{1/2 - 3 \beta/8}}^{\infty}  H_m \left(x, z \right)     \ell \left(\frac{y}{\sigma\sqrt{k}} ,\frac{z}{\sigma\sqrt{k}}\right)  dz. 
\end{align}
Since $\widetilde{H} \leq_{\ee/2} H$, 
by \eqref{DefM88} and the local limit theorem (see \eqref{LLT-general001} in Theorem \ref{LLT-general}), 
there exists $c>0$ such that for any $x \in \bb S_+^{d-1}$ and $z \in \bb R$, 
\begin{align*}
H_m \left(x, z \right) 
\leq  \bb{E}  \widetilde{H} * \kappa_{\ee}  \left( X_m^{x}, z + S_{m}^{x} \right)  
& \leq \frac{c_{\beta}}{n^{1/2 - \beta/2}} \| H * \kappa_{\ee} \|_{\nu \otimes \Leb}  
 +  \frac{c_{\ee} c_{\beta}}{ n^{(1 + \delta)/2} }  \| H * \kappa_{\ee} \|_{\mathscr H} \notag\\
 & \leq \frac{ c_{\beta} }{n^{1/2 - \beta/2}} \| H \|_{\nu \otimes \Leb}  
  +  \frac{c_{\ee} c_{\beta}}{ n^{(1 + \delta)/2} }  \| H  \|_{\mathscr H}, 
\end{align*}
where in the last inequality we used \eqref{convolution-kappa-H}. 
Using this and the fact that $m = [ n^{1-\beta}]$ and $k = n-m$ (so that $\sqrt{n/k} \leq 2$), we get
\begin{align}\label{first-term-gm-y-inte}
& \left| \int_{0}^{n^{1/2 - 3 \beta/8}}  H_m \left(x, z \right)   \ell \left(\frac{y}{\sigma\sqrt{k}} ,\frac{z}{\sigma\sqrt{k}}\right)  dz  \right|  \notag\\
& \leq   \left( \frac{c_{\beta}}{n^{1/2 - \beta/2}} \| H \|_{\nu \otimes \Leb}  +  \frac{c_{\ee} c_{\beta}}{ n^{(1 + \delta)/2} }  \| H \|_{\mathscr H} \right)
\int_{0}^{n^{1/2 - 3 \beta/8}}  \ell \left(\frac{y}{\sigma\sqrt{k}}, \frac{z}{\sigma\sqrt{k}}\right)  dz \notag\\
& \leq   \left(  c_{\beta} n^{\beta/2}  \| H \|_{\nu \otimes \Leb}  +  \frac{c_{\ee} c_{\beta}}{ n^{ \delta/2} }  \| H \|_{\mathscr H} \right)
  \int_{0}^{ 2  n^{- 3 \beta/8} }  \ell \left(\frac{y}{\sigma\sqrt{k}}, z \right)   dz,
\end{align}
where in the last line we used a change of variable $\frac{z}{\sigma\sqrt{k}} = z'$. 
Note that, by Lemma \ref{Lipschiz  property for ell_H in y} and the fact that $\ell(y, 0) = 0$, 
there exists a constant $c >0$ such that $|\ell(y, z)| = |\ell(y, z) - \ell(y, 0)| \leq c |z|$ 
for any $y, z \in \bb R$,
so that  $| \int_{0}^{ 2  n^{- 3 \beta/8} }  \ell \left(\frac{y}{\sigma\sqrt{k}}, z \right)   dz| \leq  
c  \int_{0}^{ 2  n^{- 3 \beta/8} }  z  dz  \leq c n^{-3\beta/4}$. 
Hence, from \eqref{first-term-gm-y-inte}
we obtain 
\begin{align}\label{integral-gm-ell-small}
 \left| \int_{\bb S_+^{d-1}}  \int_{0}^{n^{1/2 - 3 \beta/8}}
   H_m \left(x, z \right)   \ell \left(\frac{y}{\sigma\sqrt{k}} ,\frac{z}{\sigma\sqrt{k}}\right)  dz \, \nu(dx)  \right| 
    \leq  \frac{c_{\beta}}{ n^{\beta/4} }   \| H \|_{\nu \otimes \Leb}  +  \frac{c_{\ee} c_{\beta}}{ n^{ \delta/2} }  \| H \|_{\mathscr H}. 
\end{align}

Next, we handle the second term on the right-hand side of \eqref{decom-integral-gm-ell}. 
Since the function $(y, z) \mapsto |\ell(y, z)|$ is bounded on $\bb R \times \bb R$, 
using the stationarity of the measure $\nu$, we get
\begin{align*}
&  \left| \int_{\bb S_+^{d-1}}  \int_{ n^{1/2 - 3 \beta/8}}^{\infty}
  H_m \left(x, z \right)     \ell \left(\frac{y}{\sigma\sqrt{k}} ,\frac{z}{\sigma\sqrt{k}}\right)  dz \, \nu(dx)  \right|  \notag\\
& \leq  c  \int_{\bb S_+^{d-1}}  \int_{\bb R}  
\bb{E} \Big[ H * \kappa_{\ee} \left( x, z + S_{m}^{x} \right);   \tau_{x, z - 2\ee} \leq m  \Big]
  \mathds 1_{\{ z  \geq  n^{1/2 - 3 \beta/8}  \}}  dz \, \nu(dx)   \notag\\
& = c  \int_{\bb S_+^{d-1}}  \int_{\bb R}  H * \kappa_{\ee} \left( x,  t \right)
\bb P \left( \tau_{x, t - S_m^x - 2\ee}  \leq m,  t - S_m^x \geq  n^{1/2 - 3 \beta/8}   \right) dt \, \nu(dx), 
\end{align*}
where in the last equality we used a change of variable $z + S_m^x = t$.
Observe that the event $\{  \tau_{x, t - S_m^x - 2\ee}  \leq m \}$
means that there exists $1 \leq j \leq m$ such that 
\begin{align*}
t - S_m^x + S_j^x - 2\ee  < 0. 
\end{align*}
Thus, on the event $\{ \tau_{x, t - S_m^x - 2 \ee}  \leq m,  t - S_m^x \geq  n^{1/2 - 3 \beta/8}  \}$, 
there exists $1 \leq j \leq m$ satisfying $n^{1/2 - 3 \beta/8} - 2\ee + S_j^x <   0$. 
This implies that 
\begin{align*}
\bb P \left( \tau_{x, t - S_m^x - 2\ee}  \leq m,  t - S_m^x \geq  n^{1/2 - 3 \beta/8}  \right)
& \leq  \bb P \left(  n^{1/2 - 3 \beta/8} - \ee + \min_{1 \leq j \leq m} S_j^x   < 0   \right)  \notag\\
& \leq  \bb P \left(  \max_{1 \leq j \leq m}  |S_j^x| \geq   \frac{1}{2} n^{1/2 - 3 \beta/8}  \right). 
\end{align*}
Noting that $m= [ n^{1 - \beta} ]$, 
we use the martingale approximation (Lemma \ref{Prop-MartApp})
and \eqref{appli-Fuk-Nagaev-001} to get
\begin{align}\label{Proba_002}
\bb P \left(  \max_{1 \leq j \leq m}  |S_j^x| \geq   \frac{1}{2} n^{1/2 - 3 \beta/8}  \right)
= \bb P \left(  \max_{1 \leq j \leq m}  |S_j^x| \geq   \frac{1}{2}  n^{\beta/8} \sqrt{m}   \right)
 \leq  \frac{c_{\beta}}{ n^{\beta/8}} + \frac{c}{ n^{\delta/2} } \leq \frac{c_{\beta}}{ n^{\beta/8}}, 
\end{align}
where we take $\beta \leq \delta$. 
Hence, using \eqref{convolution-kappa-H}, we deduce that  
\begin{align*}
 \left| \int_{\bb S_+^{d-1}}  \int_{ n^{1/2 - 3 \beta/8}}^{\infty}
  H_m \left(x, z \right)     \ell \left(\frac{y}{\sigma\sqrt{k}} ,\frac{z}{\sigma\sqrt{k}}\right)  dz \, \nu(dx)  \right|  
 \leq  \frac{c_{\beta}}{ n^{\beta/8}}  \| H \|_{\nu \otimes \Leb}. 
\end{align*}
Combining this with \eqref{decom-integral-gm-ell} and \eqref{integral-gm-ell-small}, we conclude that 
\begin{align}\label{Bound_J1}
\left| \int_{\bb S_+^{d-1}}  \int_{\bb R_+}  H_m \left(x, z \right)     \ell \left(\frac{y}{\sigma\sqrt{k}} ,\frac{z}{\sigma\sqrt{k}}\right)  dz \, \nu(dx)
\right|
  \leq   \frac{c_{\beta}}{ n^{\beta/8}}  \| H \|_{\nu \otimes \Leb} +  \frac{c_{\ee} c_{\beta}}{ n^{ \delta/2} }  \| H \|_{\mathscr H}.
\end{align}
Using the second inequality in \eqref{inequ-convo-funcs},
we have $\| H_{m} \|_{ \mathscr H}  \leq  \frac{c}{\ee} m^{ \frac{4 + \delta}{ 4(2 + \delta) } }  \| H\|_{\mathscr H} 
\leq  \frac{c}{\ee} n^{ \frac{4 + \delta}{ 4(2 + \delta) } } \| H\|_{\mathscr H}$. 
Substituting this and \eqref{Bound_J1} into \eqref{eqt-A 001_Lower} yields 
\begin{align*}
I_{n,2}(x, y) & \leq  c  \frac{V_n(x,y) }{n}  \left(  \frac{c_{\beta}}{ n^{\beta/8}} \| H \|_{\nu \otimes \Leb} +
  \frac{c_{\ee} c_{\beta}}{ n^{ \delta/2} }  \| H \|_{\mathscr H} \right)  \notag\\
  & \quad + c_{\beta}  \frac{1 + V_n(x,y) }{ n^{1 + \beta} } \| H \|_{\nu \otimes \Leb} 
+  \frac{c_{\beta}}{\ee}  \frac{1 + V_n(x,y) }{ n^{1 + \frac{\delta}{2} -  \frac{4 + \delta}{ 4(2 + \delta) } - \beta }  }  \| H\|_{\mathscr H}. 
\end{align*}
Since $\frac{\delta}{2} -  \frac{4 + \delta}{ 4(2 + \delta) }  = \frac{ 2 \delta^2 + 3 \delta - 4 }{4(2 + \delta)}$, 
 together with \eqref{Lower_In_lll}, 
this completes the proof of the lower bound \eqref{eqt-A 002}. 

To finish the proof of Theorem \ref{t-B 002}, 
we note that, under the Furstenberg-Kesten condition \ref{Condi-FK} (in place of \ref{Condi-AP}), 
the second inequality in \eqref{inequ-convo-funcs} can be improved to 
$\| L_{m, \ee} \|_{ \mathscr H}  \leq  \frac{c}{\ee}   \| H\|_{\mathscr H}$.

\begin{proof}[Proof of Theorems \ref{Thm-nodrift-ysmall} and \ref{Thm-nodrift-ysmall-target}]
Theorems \ref{Thm-nodrift-ysmall} and \ref{Thm-nodrift-ysmall-target} 
are consequences of the more general statement in Theorem \ref{t-B 002}. 
\end{proof}

\section{Proof of Theorem \ref{Thm-CLLT-small-x}}

\subsection{Strict contraction properties}

In this section, we establish key technical lemmas concerning the strict contraction properties of the Markov chain $(X_n^x)_{n \geq 0}$ 
under the Furstenberg-Kesten condition \ref{Condi-FK}. 
These results will be used to obtain Theorem \ref{Thm-CLLT-small-x}. 
The following lemma shows that the matrix norm $\|g\|$ and the vector norm $|gx|$ are comparable.

\begin{lemma}\label{lemma kappa 1}
Assume \ref{Condi-FK}. 
Then, for any  $x \in \bb S_+^{d-1}$ and $\mu$-almost every $g \in \mathcal M_+$, 
\begin{align*}
\frac{1}{\varkappa^2} \|g\| \leq  | gx | \leq  \|g\|, 
\end{align*}
where $\varkappa > 1$ is the constant from condition \ref{Condi-FK}. 
\end{lemma}

\begin{proof}
The second inequality is obvious by the definition of the matrix norm 
\begin{align}\label{compute-matrix-norm}
\|g\| = \sup_{x \in \bb S_+^{d-1}} |gx| = \max_{1 \leq j \leq d} \sum_{i=1}^d g^{i,j}. 
\end{align}
For $\mu$-almost every $g = (g^{i,j})_{1 \leq i, j \leq d} \in \mathcal M_+$
 and $x = (x_1, \ldots, x_d) \in \bb S_+^{d-1}$, we have
\begin{align*}
|gx| 
= \sum_{i = 1}^d \sum_{j = 1}^d  g^{i,j}  x_j
 \geq \frac{1}{\varkappa}  \sum_{i = 1}^d \sum_{j = 1}^d  g^{i,i} x_j
=  \frac{1}{\varkappa}  \sum_{i = 1}^d  g^{i,i}  
 \geq  \frac{1}{\varkappa^2}  \max_{1 \leq j \leq d}  \sum_{i = 1}^d  g^{i,j}  = \frac{1}{\varkappa^2}  \|g\|, 
\end{align*}
which proves the first inequality. 
\end{proof}

The following lemma relies on the approximate duality relation \eqref{intro-duality} and plays an important role in the sequel. 

\begin{lemma}\label{Cor_CoarseBound}
Assume \ref{CondiMoment} and \ref{Condi-FK}. 
Then there exists a constant $c>0$ such that for any measurable function $H$ on 
$\bb S_+^{d-1} \times \bb R$, $n \geq 1$ and $x' \in \bb S_+^{d-1}$, 
\begin{align*}
I: = \int_{\bb R} \sup_{x \in \bb S_+^{d-1}} \bb E  \Big[ H( X_n^x,  y + S_n^x);  \tau_{x, y} > n  \Big]  dy
   \leq \frac{c}{\sqrt{n}}  \int_{ 0 }^{\infty}  H_1 (z)  (1 + V_n^*(x', z))  dz,
\end{align*}
where $\overline{\varkappa} = 2 \log \varkappa$ 
and $H_1(z) = \sup_{y \in \bb R: |y-z| \leq  2\overline{\varkappa}}  \sup_{x \in \bb S_+^{d-1}} |H(x, y)|$, $z \in \bb R$. 
\end{lemma}

\begin{proof}
By Lemma \ref{lemma kappa 1}, it holds that for any $x, x' \in \bb S_+^{d-1}$ and $\mu$-almost every $g \in \mathcal M_+$, 
\begin{align}\label{inequa-log-norm-vec}
\log |gx| \geq \log \|g\| - \overline{\varkappa} \geq  \log |gx'| - \overline{\varkappa}, 
\end{align}
where $\overline{\varkappa} = 2 \log \varkappa$ and $\varkappa > 1$ is given by condition \ref{Condi-FK}. 
Hence, for any $y \in \bb R$ and $x, x' \in \bb S_+^{d-1}$,
\begin{align*}
 \bb E  \Big[ H( X_n^x,  y + S_n^x);  \tau_{x, y} > n  \Big]  
 \leq  \bb E  \Big[ \overline{H} (y + S_n^{x'} );   \tau_{x', y + \overline{\varkappa}} > n  \Big],  
\end{align*}
where $\overline{H}(z) = \sup_{y \in \bb R: |y-z| \leq  \overline{\varkappa}}  \sup_{x \in \bb S_+^{d-1}} |H(x, y)|$, $z \in \bb R$. 
By Fubini's theorem and a change of variable $y + \overline{\varkappa} +  S_n = z$, we get 
\begin{align}\label{bound-I-hh}
I:&  \leq  \int_{\bb R}  \bb E  \Big[ \overline{H} (y + S_n^{x'});   \tau_{x', y + \overline{\varkappa}} > n  \Big]   dy  \notag\\
& =  \bb E  \int_{\bb R}  \overline{H} (y + S_n^{x'})  
   \mathds 1_{ \{ y + \overline{\varkappa} + S_1^{x'} \geq 0,  \ldots, y + \overline{\varkappa} + S_n^{x'} \geq 0  \} }  dy   \notag\\
& =    \int_{\bb R}  \overline{H} (z - \overline{\varkappa})  
   \bb P \left( z - S_n^{x'} + S_1^{x'} \geq 0,  \ldots, 
                   z  - S_n^{x'} + S_{n-1}^{x'} \geq 0,  z  \geq 0  \right)  dz   \notag\\
& =    \int_{0}^{\infty}  \overline{H} (z - \overline{\varkappa}) 
   \bb P \left( z  - S_n^{x'} + S_1^{x'} \geq 0,  \ldots, 
                   z  - S_n^{x'} + S_{n-1}^{x'} \geq 0  \right)  dz. 
\end{align}
Using  \eqref{inequa-log-norm-vec}, 
we have that under $\bb P$, for any $1 \leq k \leq n-1$, 
\begin{align}\label{inequ-BJ01}
- S_n^{x'} + S_k^{x'} 
  = - \log | g_n \cdots g_{k+1} ( (g_k \cdots g_1) \cdot x') | 
 \leq  - \log \| g_n \cdots g_{k+1} \|  +  \overline{\varkappa}. 
\end{align}
As a consequence of \eqref{compute-matrix-norm}, it holds that, $\mu$-almost every $g \in \mathcal M_+$, 
\begin{align}\label{inequa-g-transpose}
\| g^{ {\rm T} } \| 
= \max_{1 \leq i \leq d} \sum_{j=1}^d g^{i,j} 
\leq  d \max_{1 \leq j \leq d} \sum_{i=1}^d g^{i,j}
=  d \|g\|. 
\end{align}
By \eqref{inequ-BJ01}, \eqref{inequa-g-transpose}
 and  Lemma \ref{lemma kappa 1},  for any $x'' \in \bb S_+^{d-1}$ and $1 \leq k \leq n-1$, 
 we have that under $\bb P$, 
\begin{align}\label{inequ-dual-Sn}
- S_n^{x''} + S_k^{x''}  \leq  - \log \| g_{k+1}^{ {\rm T} }  \cdots g_n^{ {\rm T} }  \|  +  \overline{\varkappa} + \log d
\leq  - \log | g_{k+1}^{ {\rm T} }  \cdots g_n^{ {\rm T} }  x''|  +  \overline{\varkappa} + \log d. 
\end{align}
By \eqref{def-Xn-Sn-dual}, for $1 \leq i \leq n$, we have denoted $h_i = g_{n-i+1}^{ {\rm T} }$ 
and $S_i^{x'', *} = - \log |h_i \cdots h_1 x''|$. 
It follows that under $\bb P$, 
\begin{align}\label{inequ-dual-Sn-bb}
- S_n^{x''} + S_k^{x''}  
\leq  - \log | h_{n-k} \cdots h_1 x'' | +  \overline{\varkappa} + \log d  = S_{n-k}^{x'', *} + \overline{\varkappa} + \log d.  
\end{align}
Therefore, 
applying \eqref{bound-tau-y-bis} for $\tau^*_{x,y}$ instead of $\tau_{x,y}$, we get that there exists a constant $c>0$ such that for any 
$x'' \in \bb S_+^{d-1}$, $z \in \bb R_+$ and $n \geq 1$, 
\begin{align*}
& \bb P \left( z  - S_n^{x'} + S_1^{x'} \geq 0,  \ldots, 
                   z  - S_n^{x'} + S_{n-1}^{x'} \geq 0  \right)   \notag\\
& \leq  \bb P \left( z + \overline{\varkappa} + \log d + S_{n-1}^{x'', *} \geq 0,  \ldots, 
                   z + \overline{\varkappa} + \log d + S_{1}^{x'', *} \geq 0  \right)   \notag\\
& = \bb P \left( \tau_{x'',  z + \overline{\varkappa} + \log d }^* > n-1 \right) 
  \leq  c \frac{1 + V_n^*(x'', z + \overline{\varkappa} + \log d) }{\sqrt{n}} 
  \leq  c' \frac{1 + V_n^*(x'', z) }{\sqrt{n}} . 
\end{align*}
Substituting this into \eqref{bound-I-hh} and recalling that 
$H_1(z) = \sup_{y \in \bb R: |y-z| \leq  2\overline{\varkappa}}  \sup_{x \in \bb S_+^{d-1}} |H(x, y)|$,
we conclude the proof of the lemma. 
\end{proof}

Under condition \ref{Condi-AP}, 
it follows from \cite[Lemma 3.1]{HH08} that,
for any $a \in (0,1)$, there exists a constant $c = c(a) > 0$ 
such that for any $x_1, x_2 \in \bb S_+^{d-1}$ with $\mathbf{d}(x_1, x_2) < a$, 
and $\mu$-almost every $g \in \mathcal M_+$, 
\begin{align}\label{Lem-Contin-Cocycle}
\left|  \log |g x_1| - \log |g x_2| \right| 
\leq  c \mathbf{d}(x_1, x_2). 
\end{align}
Using \eqref{Lem-Contin-Cocycle}, we now establish the following analogue of 
 Lemma \ref{Lem-conti-square} when condition \ref{Condi-AP} is replaced by \ref{Condi-FK}.  

\begin{lemma}\label{Lem-contractivity}
Assume \ref{CondiMoment} and \ref{Condi-FK}.   
Then there exists  $c >0$  such that, 
  $\bb P$-almost surely,
\begin{align}\label{contractivity-inequ-0}
\sup_{n \geq 1} \sup_{x, x' \in \bb S_+^{d-1}: x \neq x'}  \frac{|S_n^x -   S_n^{x'}|}{ \bf d(x, x') }   \leq  c, 
\end{align}
and
\begin{align}\label{contractivity-inequ}
\sup_{n \geq 1} \sup_{x, x' \in \bb S_+^{d-1}: x \neq x'} 
\frac{ \left| \min_{1 \leq j \leq n}   S_j^x -  \min_{1 \leq j \leq n}  S_j^{x'} \right| }{ \bf d(x, x') }
\leq  c. 
\end{align}
\end{lemma}

\begin{proof}
We only prove \eqref{contractivity-inequ} since the proof of \eqref{contractivity-inequ-0} is similar and easier. 

We first consider the case when $\bf d(x, x') > \frac{1}{2}$. 
By Lemma \ref{lemma kappa 1},  $\bb P$-almost surely, for any $x, x' \in \bb S_+^{d-1}$ and $j \geq 1$, 
\begin{align*}
S_j^x \leq  S_j^{x'} + 2 \log \varkappa,
\end{align*}
where $\varkappa > 1$ is given by condition \ref{Condi-FK}. 
Taking the minimum over $1 \leq j \leq n$ first on the left hand side and then on the right hand side, 
we get that, $\bb P$-almost surely, for any $x, x' \in \bb S_+^{d-1}$, 
\begin{align*}
\min_{1 \leq j \leq n}  S_j^x  \leq  \min_{1 \leq j \leq n} S_j^{x'}  + 2 \log \varkappa.  
\end{align*}
The lower bound for $\min_{1 \leq j \leq n}  S_j^x$ can be obtained in the same way.
Therefore,  when $\bf d(x, x') > \frac{1}{2}$,  \eqref{contractivity-inequ} holds with $c = 4 \log \varkappa$.

We next consider the case when $\bf d(x, x') \leq \frac{1}{2}$. 
Under condition \ref{Condi-FK}, there exists a constant $0< r <1$ such that, $\bb P$-almost surely,
 for any $x, x' \in \bb S_+^{d-1}$ and $k \geq 1$, 
\begin{align}\label{contractivity-MC}
\bf d \left( X_{k-1}^x, X_{k-1}^{x'} \right)  \leq  r^{k-1}  \bf d \left(x, x' \right), 
\end{align}
see \cite[Proposition 3.1]{Hen97}. 
Hence, by \eqref{Lem-Contin-Cocycle}, 
\begin{align*}
\Big| \log |g_k X_{k-1}^x| - \log |g_k X_{k-1}^{x'}|  \Big|
\leq  c r^{k-1}  \bf d \left(x, x' \right). 
\end{align*}
Since $S_j^x = \sum_{k = 1}^j \log |g_k X_{k-1}^x|$ and $X_0^x = x$, 
it follows that for any $1 \leq j \leq n$ and $x, x' \in \bb S_+^{d-1}$, 
\begin{align*}
\min_{1 \leq j \leq n}  S_j^x  
\leq  S_j^x  
 \leq  S_j^{x'} + c \sum_{k=1}^n  r^{k-1}  \bf d \left(x, x' \right)
 \leq  S_j^{x'}  + c' \bf d \left(x, x' \right).
\end{align*}
Taking the minimum over $1 \leq j \leq n$, we get that,
$\bb P$-almost surely, for any $x, x' \in \bb S_+^{d-1}$, 
\begin{align*}
\min_{1 \leq j \leq n}  S_j^x 
 \leq  \min_{1 \leq j \leq n}  S_j^{x'} + c \bf d \left(x, x' \right).  
\end{align*}
The lower bound can be obtained in the same way. 
\end{proof}

The following result is analogous to Lemma \ref{Lem_Inequality_Aoverline} but achieves a faster convergence rate due to condition \ref{Condi-FK}. This result will be essential for applying Theorem \ref{t-B 002}, particularly in handling specific functions arising in the proof of Theorem \ref{Thm-CLLT-small-x}.

Let $\rho$ be a nonnegative smooth  function with compact support in $[-1, 1]$, normalized in such a way that
$\int_{-1}^1 \rho(u)du=1$. For $\ee >0$,  define the rescaled function $\rho_{\ee}(u) = \frac{1}{\ee} \rho(\frac{u}{\ee})$, $u \in \bb R$.
Recall that $\chi_{\ee}$ is defined by \eqref{Def_chiee}.

\begin{lemma}\label{Lem_HolderNormPsi}
Assume \ref{CondiMoment} and \ref{Condi-FK}. 
For $(x,y) \in \bb S_+^{d-1} \times \bb R$, $m \geq 1$, $\ee >0$ and $H \in \scr H$, define 
\begin{align*}
\overline H_{m,\ee}(x, y) :  = 
\bb E  \left[  H * \rho_{\ee} \left( X_m^x,  y + S_m^x  \right)    
   \chi_{\ee}  \left( y + \ee + \min_{1 \leq  j \leq m}   S_j^x  \right)  \right]. 
\end{align*}
Then $\overline H_{m,\ee} \in \scr H$ 
and there exist  $0< r < 1$, $c_0 >0$ and $c_{\ee} > 0$ such that for any $x' \in \bb S_+^{d-1}$, 
\begin{align*}
\| \overline H_{m,\ee} \|_{\scr H} 
\leq  \frac{c_{\ee}}{ \sqrt{m}} \int_{ \bb R}  H_1 (z)  (1 + V_n^*(x', z))   dz  +  c_{\ee}   r^m   \|H\|_{\scr H},
\end{align*}
where 
$H_1(z) = \sup_{y \in \bb R: |y-z| \leq  c_0}  \sup_{x \in \bb S_+^{d-1}} |H(x, y)|$, $z \in \bb R$. 
\end{lemma}

\begin{proof}
Without loss of generality, we assume that $H$ is nonnegative. 
 Recall that 
\begin{align}\label{Holder-Hmee-abc}
\|  \overline H_{m,\ee} \|_{\mathscr H}
 =  \int_{\bb R}   \sup_{x \in \bb S_+^{d-1}}  | \overline H_{m,\ee} \left(x, y \right)|  dy
   +  \int_{\bb R} \sup_{x, x' \in \bb S_+^{d-1}: x \neq x'}  
    \frac{| \overline H_{m,\ee} \left(x, y \right)  - \overline H_{m,\ee} \left(x', y \right)  |}{ \bf d(x, x') } dy. 
\end{align}
For the first term in \eqref{Holder-Hmee-abc}, 
by the definition of $\chi_{\ee}$ (cf.\ \eqref{Def_chiee}), we have 
\begin{align*}
\int_{\bb R}   \sup_{x \in \bb S_+^{d-1}}  | \overline H_{m,\ee} \left(x, y \right)|  dy
& \leq  \int_{\bb R}   \sup_{x \in \bb S_+^{d-1}}  \bb E  \left[  H * \rho_{\ee} \left( X_m^x,  y + S_m^x  \right)    
    \mathds 1_{ \{ \tau_{x, y + 2 \ee} > m \} }  \right]   dy   \notag\\
& =   \int_{\bb R}   \sup_{x \in \bb S_+^{d-1}}  \bb E  \left[  H * \rho_{\ee} \left( X_m^x,  y' - 2 \ee + S_m^x  \right)    
    \mathds 1_{ \{ \tau_{x, y'} > m \} }  \right]   dy'    \notag\\
& \leq  \int_{\bb R}   \sup_{x \in \bb S_+^{d-1}}  \bb E  \left[  h_{\ee} \left( y'  + S_m^x  \right)    
    \mathds 1_{ \{ \tau_{x, y'} > m \} }  \right]   dy', 
\end{align*}
where $h_{\ee}(u) = \sup_{u\in \bb R: |u-y| \leq 3\ee} \sup_{x \in \bb S_+^{d-1}} H(x, y)$
and we used the fact that $H * \rho_{\ee} ( X_m^x,  y' - 2 \ee + S_m^x ) \leq h_{\ee}(y' + S_m^x)$. 
By Lemma \ref{Cor_CoarseBound}, it follows that for any $x' \in \bb S_+^{d-1}$, 
\begin{align}\label{bound-sup-norm-aaa}
\int_{\bb R}   \sup_{x \in \bb S_+^{d-1}}  | \overline H_{m,\ee} \left(x, y \right)|  dy
\leq  \frac{c}{\sqrt{m}}  \int_{0}^{\infty}  H_1 (z)  (1 + V_n^*(x', z))  dz. 
\end{align}
For the second term in \eqref{Holder-Hmee-abc}, since the function $\chi_{\ee}$ is $1/\ee$-Lipschitz continuous on $\bb R$, 
applying Lemma \ref{Lem-contractivity} twice, we get that 
there exist constants $c_0, c > 0$ such that,  for any $\ee > 0$, $x, x' \in \bb S_+^{d-1}$, $y \in \bb R$ and $m \geq 1$,
\begin{align*}
I_m(x, x', y):& =  \left|  \chi_{\ee}  \left( y + \ee + \min_{1 \leq  j \leq m}   S_j^x \right) 
        - \chi_{\ee}  \left( y + \ee+ \min_{1 \leq  j \leq m}  S_j^{x'} \right)   \right|   \notag\\
& = \left| \chi_{\ee}  \left( y + \ee + \min_{1 \leq  j \leq m}  S_j^x \right) 
       - \chi_{\ee}  \left( y + \ee  + \min_{1 \leq  j \leq m}  S_j^{x'} \right)  \right|
     \mathds 1_{\{ y  + \min_{1 \leq  j \leq m}  S_j^x \geq - c_0 \}}   \notag\\
& \leq \frac{1}{\ee} 
\left|  \min_{1 \leq  j \leq m}  S_j^x - \min_{1 \leq  j \leq m}  S_j^{x'} \right| 
\mathds 1_{\{ y  + \min_{1 \leq  j \leq m}  S_j^x \geq - c_0 \}}    \notag\\
& \leq  \frac{c}{\ee}  \bf d(x, x')  \mathds 1_{\{ y  + \min_{1 \leq  j \leq m}  S_j^x \geq - c_0 \}}, 
\end{align*}
$\bb P$-a.s. 
By Lemma \ref{Cor_CoarseBound}, there exist constants $c, c_0, c_1 >0$ such that for any $\ee > 0$ and $m \geq 1$, 
\begin{align}\label{HolderBound1}
& \int_{\bb R} \sup_{x, x' \in \bb S_+^{d-1}: x \neq x'}  \frac{1}{ \bf d(x,x') }
\bb E  \Big[  H * \rho_{\ee} \left( X_m^x, y + S_m^x \right)  I_m(x, x', y)   \Big] dy  \notag\\
 & \leq   \frac{c}{\ee}  \int_{\bb R} \sup_{x \in \bb S_+^{d-1}} \bb E  \Big[ H( X_n^x,  y + S_n^x);  \tau_{x, y + c_0} > n  \Big]  dy   \notag\\
 & =   \frac{c}{\ee}  \int_{\bb R} \sup_{x \in \bb S_+^{d-1}} \bb E_x  \Big[ H( X_n^x,  y - c_0 + S_n^x);  \tau_{x, y} > n  \Big]  dy   \notag\\
 & \leq \frac{c_1}{\ee \sqrt{m}}  \int_{ 0 }^{\infty}  H_1 (z - c_0)  (1 + V_n^*(x', z))  dz.  
\end{align}
Recall that $\rho_{\ee}$ is a smooth function supported on $[-\ee, \ee]$. 
For any $y \in \bb R$, let $\rho_{\ee, 1}(y) = \sup_{s \in \bb R: |s| \leq \ee} |\rho_{\ee} (y + s)|$ 
and $\rho_{\ee, 2}(y) = \sup_{s \in \bb R: |s| \leq \ee} |\rho_{\ee}' (y + s)|$.
Then we have that for any $x, x' \in \bb S_+^{d-1}$, $y \in \bb R$ and $m \geq 1$, 
\begin{align*}
J_m(x, x', y):& =    \left| H * \rho_{\ee} (X_m^x, y + S_m^x)   -  H * \rho_{\ee} (X_m^{x'}, y + S_m^{x'})  \right|   \notag\\
& =   \left|  \int_{\bb R}  \rho_{\ee} (y + S_m^x - u) H(X_m^x, u) du 
      -  \int_{\bb R}  \rho_{\ee} (y + S_m^{x'} - u) H(X_m^{x'}, u) du \right|  \notag\\
 & \leq   \int_{\bb R}  \Big| \rho_{\ee} (y + S_m^{x} - u)  -  \rho_{\ee} (y + S_m^{x'} - u)  \Big|   H(X_m^{x'}, u) du
       \notag\\
 & \quad   +   \int_{\bb R}  \rho_{\ee} (y + S_m^x - u)  \Big| H(X_m^x, u) -  H(X_m^{x'}, u)  \Big|  du   \notag\\
& \leq  \rho_{\ee, 2}(y)   \left| S_m^x - S_m^{x'} \right|  \int_{\bb R} H(X_m^{x'}, u) du
    +  \rho_{\ee, 1}(y) \bf d \left( X_m^x, X_m^{x'}  \right)   \int_{\bb R} \| H(\cdot, u) \|_{\scr B} du, 
\end{align*}
$\bb P$-almost surely. 
By \eqref{contractivity-inequ-0} and \eqref{contractivity-MC}, 
there exist constants $1 < r < 1$ and $c>0$ such that 
for any $\ee >0$, $x, x' \in \bb S_+^{d-1}$, $y \in \bb R$ and $m \geq 1$, 
\begin{align*}
J_m(x, x', y)  
 \leq  c \, \rho_{\ee, 2}(y)   \bf d (x, x')  \int_{\bb R} H_1(u) du + \rho_{\ee, 1}(y)  r^m  \bf d (x, x')  \|H\|_{\scr H},  
\end{align*}
$\bb P$-almost surely. 
Therefore, using the fact that $\chi_{\ee} \leq 1$ and \eqref{bound-tau-y-bis}, we obtain
that there exist constants $1 < r < 1$ and $c>0$ such that 
for any $\ee >0$, $x, x' \in \bb S_+^{d-1}$, $y \in \bb R$ and $m \geq 1$, 
\begin{align*}
&  \bb E  \left[ J_m(x, x', y)    \chi_{\ee}   \left( y + \ee + \min_{1 \leq  j \leq m}  S_j^{x'} \right)  \right]    \notag\\
& \leq  \frac{c(1 + \max\{y, 0\})}{ \sqrt{m} }  \rho_{\ee, 2}(y)   \bf d (x, x')  \int_{\bb R} H_1(u) du
   + \rho_{\ee, 1}(y)  r^m  \bf d (x, x')  \|H\|_{\scr H}. 
 \end{align*}
Consequently, there exists a constant $1 < r < 1$ such that the following holds: for any $\ee >0$, there exists $c_{\ee} > 0$ such that 
for any  $m \geq 1$, 
\begin{align}\label{HolderBound2}
& \int_{\bb R}  \sup_{x, x' \in \bb S_+^{d-1}: x \neq x'}  \frac{1}{ \bf d(x, x') }  
        \bb E  \left[ J_m(x, x', y)    \chi_{\ee}   \left( y + \ee + \min_{1 \leq  j \leq m}  S_j^{x'} \right)  \right]  dy  \notag\\
& \leq  \frac{c_{\ee}}{ \sqrt{m} }   \int_{\bb R} H_1(u) du  +  c_{\ee}   r^m   \|H\|_{\scr H}. 
\end{align}
Putting together \eqref{bound-sup-norm-aaa}, \eqref{HolderBound1} and \eqref{HolderBound2} concludes the proof of the lemma. 
\end{proof}

\subsection{Proof of Theorem \ref{Thm-CLLT-small-x}}
Let $\rho$ be a nonnegative smooth function supported on $[-1, 1]$ such that
$\int_{-1}^1 \rho(u)du=1$ and set $\rho_{\ee}(u) = \frac{1}{\ee} \rho(\frac{u}{\ee})$ for $\ee >0$.
Since  $F \leq_{\ee} H \leq_{\ee} \widetilde{H}$, 
it holds that $F \leq  H * \rho_{\ee} \leq \widetilde{H}$. 
Set $m=\left[ n/2 \right]$ and $k = n-m.$ 
By the Markov property of the pair $(X_n^x, S_n^x)$, we have that for any $x \in \bb S_+^{d-1}$, $y \in \mathbb R$, 
$z \in \bb R_+$ and $\Delta \geq \Delta_0 >0$, 
\begin{align} \label{JJJ-markov property_aa}
I_n(x, y) :& = \bb P \Big(  y + S_n^x \in [z, z+ \Delta],  \tau_{x, y} > n  \Big)    \notag\\
& = \int_{\bb S_+^{d-1} \times \mathbb R_+}  I_m(x', y') \bb P \left(X_k^x \in dx', y + S_k^x \in dy',  \tau_{x, y} > k \right). 
\end{align}
When $y<0$, we set $I_n(x, y)=0$ for any $x \in \bb S_+^{d-1}$ and $n \geq 1$. 
Now we want to find an $\ee$-dominating function for $y' \mapsto I_m(x', y')$ 
given on the right hand side of \eqref{JJJ-markov property_aa}: 
for any $x' \in \bb S_+^{d-1}$, $y' \in \mathbb R_+$ and $|v|\leq \ee$, 
\begin{align}\label{DefImtbb}  
I_m(x', y') =  \bb E \Big[ F(y' + S_m^{x'} );  \tau_{x', y'} > m  \Big]  \leq H_{m,\ee}(x', y' + v), 
\end{align}
where 
\begin{align} \label{DefHmt}  
H_{m,\ee}(x', y') := \bb E \left[ H * \rho_{\ee} (y' + S_m^{x'}) \chi_{\ee}  \left( y' + \ee + \min_{1 \leq  j \leq m}  S_j^{x'} \right)  \right]
\end{align}
with $\chi_{\ee}$ defined by \eqref{Def_chiee}. 
Since $F \leq_{\ee} H$, 
it is easy to see that $I_m \leq_{\ee} H_{m, \ee}$ and that both $I_m$ and $H_{m, \ee}$ are integrable functions on $\bb S_+^{d-1} \times \bb R$. 
By the upper bound \eqref{last-inequality-001}, 
there exists a constant $\beta_0 >0$ such that
for any $\beta \in (0, \beta_0)$, $x \in \bb S_+^{d-1}$ and $y \in \mathbb R_+$, 
\begin{align}\label{UpperBoundhhn32}
I_n(x, y)   &  \leq   c  \frac{1 + V_n(x, y) }{ n } \| H_{m,\ee} \|_{\nu \otimes \Leb} 
+  c_{\beta} \frac{1 + V_n(x, y) }{ n^{1 + \frac{\delta}{2} - \beta} }  \| H_{m,\ee} \|_{\mathscr H}. 
\end{align}
Using a change of variable $y' + S_m = z$ and the fact that $H * \rho_{\ee} \leq \widetilde{H}$, we get
\begin{align}\label{Equa-J1-aa}
\| H_{m,\ee} \|_{\nu \otimes \Leb} & = 
    \int_{\bb S_+^{d-1} \times \bb R } 
    \bb E \left[ H * \rho_{\ee} (y' + S_m^{x'}) \chi_{\ee}  \left( y' + \ee + \min_{1 \leq  j \leq m}  S_j^{x'} \right)  \right]
        \nu(dx')  dy'  \notag\\
 & = 
    \int_{\bb S_+^{d-1} \times \bb R }   H * \rho_{\ee} (z)
    \bb E \left[  \chi_{\ee}  \left( z - S_m^{x'} + \ee + \min_{1 \leq  j \leq m}  S_j^{x'}  \right) 
         \right]    \nu(dx')  dz   \notag\\
  & \leq  
    \int_{\bb S_+^{d-1} \times \bb R }   \widetilde{H} (z)
    \bb E \left[  \chi_{\ee}  \left( z - S_m^{x'} + \ee + \min_{1 \leq  j \leq m}  S_j^{x'} \right) 
       \right]    \nu(dx')  dz. 
\end{align}
As in \eqref{inequ-dual-Sn} and \eqref{inequ-dual-Sn-bb}, 
we use condition \ref{Condi-FK} and Lemma \ref{lemma kappa 1} to get that,
$\bb P$-a.s., 
 for any $0 \leq k \leq m$, 
\begin{align}\label{Inequa-dual-Sn}
S_{m-k}^{x', *} - \gamma  \leq  - S_m^{x'} + S_k^{x'}  \leq   S_{m-k}^{x', *} + \gamma, 
\end{align}
where $\gamma = 2 \log \varkappa + \log d$ with $\varkappa$ given in condition \ref{Condi-FK},
and $S_i^{x', *}$ is defined by \eqref{def-Xn-Sn-dual}. 
By \eqref{Def_chiee} and \eqref{Inequa-dual-Sn}, we have that for any $z \in \bb R$ and $m \geq 1$, 
\begin{align}\label{bound-chi-ee}
\chi_{\ee}  \left( z - S_m^{x'} + \ee + \min_{1 \leq  j \leq m}  S_j^{x'} \right)
& \leq  \mathds 1_{ \big\{ z + 2\ee - S_m + \min_{1 \leq  j \leq m}  S_j^{x'}  \geq  0  \big\} }   \notag\\
& \leq  \mathds 1_{ \big\{ z + 2 \ee \geq 0,  z + 2 \ee + \gamma + S_1^{x', *}  \geq 0,  
 \ldots,  z + 2 \ee + \gamma + S_{m-1}^{x', *} \geq 0  \big\} }
   \notag\\
& \leq   \mathds 1_{ \big\{ z + 2 \ee \geq 0  \big\} }  \mathds 1_{ \big\{  \tau_{x', z + 2 \ee + \gamma }^* > m - 1 \big\} }  \notag\\
& =  \mathds 1_{ \big\{ z + 2 \ee \geq 0  \big\} }  \mathds 1_{ \big\{  \tau_{x', z  + \overline{\gamma} }^* > m - 1 \big\} }, 
\end{align}
$\bb P$-a.s., 
where $\overline{\gamma} = 2 \ee + \gamma$ and $\tau^*_{x, y}$ is defined by \eqref{def-tau-y-star}. 
Substituting \eqref{bound-chi-ee}  into \eqref{Equa-J1-aa} yields
\begin{align}\label{J-1-upper-abc}
\| H_{m,\ee} \|_{\nu \otimes \Leb}  
& \leq   \int_{-2 \ee}^{\infty}  \int_{\bb S_+^{d-1}}   \widetilde{H} (z)
    \bb P \left(   \tau_{x', z + \overline{\gamma} }^* > m  \right)    \nu(dx')  dz   \notag\\
 & =  \int_{-2 \ee + \overline{\gamma}}^{\infty}  \int_{\bb S_+^{d-1}}   \widetilde{H} (z' - \overline{\gamma})
    \bb P \left(   \tau_{x', z' }^* > m  \right)    \nu(dx')  dz'   \notag\\
 & \leq  c \int_{-2 \ee + \overline{\gamma}}^{\infty}  \int_{\bb S_+^{d-1}}  \widetilde{H} (z' - \overline{\gamma})
  \frac{1 + V_m^*(x', z')}{\sqrt{m}} \nu(dx')  dz'. 
\end{align}

For the last term in \eqref{UpperBoundhhn32}, 
by Lemma \ref{Lem_HolderNormPsi},
 there exists  a constant $0 < r < 1$ such that  for any $x' \in \bb S_+^{d-1}$, 
\begin{align}\label{second-term-pf-thm3}
\| H_{m,\ee} \|_{\mathscr H} \leq   c_{\ee}  \frac{ 1+ V_n(x, y) }{ \sqrt{n} } \int_{\bb R}  H_1 (z)  (1 +  V_n^*(x', z))   dz
    + c_{\ee}   r^n  \|H\|_{\scr H}. 
\end{align}
Substituting \eqref{J-1-upper-abc} and \eqref{second-term-pf-thm3} 
into \eqref{UpperBoundhhn32}, we conclude the proof of  Theorem \ref{Thm-CLLT-small-x}.


\end{document}